\documentclass[10pt,a4paper,english]{article}
\usepackage[T1]{fontenc}
\usepackage[latin9]{inputenc}
\usepackage{babel}
\usepackage{amsthm}
\usepackage{amsmath}
\usepackage{amssymb}
\usepackage{esint}
\usepackage[unicode=true,pdfusetitle,
 bookmarks=true,bookmarksnumbered=false,bookmarksopen=false,
 breaklinks=false,pdfborder={0 0 1},backref=false,colorlinks=false]
 {hyperref}

\makeatletter

\pdfpageheight\paperheight
\pdfpagewidth\paperwidth

\numberwithin{equation}{section}
\numberwithin{figure}{section}
\theoremstyle{plain}
\newtheorem{thm}{\protect\theoremname}
 \newcommand\thmsname{\protect\theoremname}
 \newcommand\nm@thmtype{theorem}
 \theoremstyle{plain}
 
 \newenvironment{namedthm}[1][Undefined Theorem Name]{
   \ifx{#1}{Undefined Theorem Name}\renewcommand\nm@thmtype{theorem*}
   \else\renewcommand\thmsname{#1}\renewcommand\nm@thmtype{namedtheorem}
   \fi
   \begin{\nm@thmtype}}
   {\end{\nm@thmtype}}
 \theoremstyle{definition}
 \newtheorem*{defn*}{\protect\definitionname}
\newenvironment{lyxlist}[1]
{\begin{list}{}
{\settowidth{\labelwidth}{#1}
 \setlength{\leftmargin}{\labelwidth}
 \addtolength{\leftmargin}{\labelsep}
 }}
{\end{list}}
  \theoremstyle{plain}
  \newtheorem{lem}[thm]{\protect\lemmaname}
  \theoremstyle{plain}
  
  \theoremstyle{remark}
  \newtheorem{rem}[thm]{\protect\remarkname}
  \theoremstyle{definition}
  \newtheorem{defn}[thm]{\protect\definitionname}
  \theoremstyle{plain}
  \newtheorem{prop}[thm]{\protect\propositionname}

\usepackage{a4wide}

\usepackage{type1cm}
\usepackage[T1]{fontenc}
\usepackage{courier}
\usepackage{sectsty}
\allsectionsfont{\usefont{OT1}{ptm}{bc}{n}\selectfont}
\usepackage{bbold}
\usepackage{calrsfs}

\renewenvironment{proof}[1][\proofname]{\par
      \pushQED{\qed}%
      \normalfont \topsep6\p@\@plus6\p@\relax
      \trivlist
      \item[\hskip\labelsep
            \bfseries
        #1\@addpunct{.}]\ignorespaces
    }{%
  \popQED\endtrivlist\@endpefalse
}

\makeatother

  \providecommand{\corollaryname}{Corollary}
  \providecommand{\definitionname}{Definition}
  \providecommand{\lemmaname}{Lemma}
  \providecommand{\propositionname}{Proposition}
  \providecommand{\remarkname}{Remark}
  \providecommand{\theoremname}{Theorem}
\providecommand{\theoremname}{Theorem}

\begin{document}

\title{Floer homology for magnetic fields with at most linear growth on the universal
cover}

\author{Urs Frauenfelder, Will J. Merry and Gabriel P. Paternain}
\maketitle
\begin{abstract}
The Floer homology of a cotangent bundle is isomorphic to loop space
homology of the underlying manifold, as proved by Abbondandolo-Schwarz,
Salamon-Weber, and Viterbo. In this paper we show that in the presence
of a Dirac magnetic monopole which admits a primitive with at most linear
growth on the universal cover, the Floer homology in atoroidal free
homotopy classes is again isomorphic to loop space homology. As a
consequence we prove that for any atoroidal free homotopy class and
any sufficiently small $\tau>0$, any magnetic flow associated to
the Dirac magnetic monopole has a closed orbit of period $\tau$ belonging
to the given free homotopy class. In the case where the Dirac magnetic
monopole admits a bounded primitive on the universal cover we also
prove the Conley conjecture for Hamiltonians that are quadratic at
infinity, i.e., we show that such Hamiltonians have infinitely many
periodic orbits.
\end{abstract}

\section{Introduction}

We are interested in Hamiltonian systems of the following form. The
configuration space $M$ is a closed connected oriented manifold of
dimension $n\geq2$. The Hamiltonian $H$ is a smooth function on
the phase space $T^{*}M$ which might in addition depend periodically
on time. The Dirac magnetic monopole is a closed two-form $\sigma\in\Omega^{2}(M)$
which gives rise to a twisted symplectic form \cite{ArnoldGivental1990,Ginzburg1996}
on the cotangent bundle 
\[
\omega_{\sigma}=d\lambda+\pi^{*}\sigma,
\]
 where $\lambda$ is the Liouville one-form on $T^{*}M$ and $\pi\colon T^{*}M\to M$
is the footpoint projection. The flow is generated by the time dependent
Hamiltonian vector field $X_{H,\sigma}$ defined implicitly by the
equation 
\[
-dH=\omega_{\sigma}(X_{H,\sigma},\cdot).
\]
Floer's semi-infinite dimensional Morse homology associates to a Hamiltonian
system a chain complex which is generated by the periodic orbits of
a given fixed period $\tau>0$ and a given free homotopy class $\alpha\in[S^{1},M]$,
and defines a boundary operator by counting perturbed holomorphic
cylinders which asymptotically converge to the periodic orbits. A
priori it is far form obvious that this recipe gives a well defined
boundary operator. Indeed, the question if Floer's boundary operator
is well-defined or not depends on a difficult compactness result for
the perturbed holomorphic curve equation. Tentatively we write $HF_{*}^{\alpha}(H,\sigma,\tau)$
for the Floer homology with the $\tau$-periodic Hamiltonian $H$,
the magnetic monopole $\sigma$, and the free homotopy class $\alpha$.
In order to avoid discussions about orientations of moduli spaces
we take coefficients in $\mathbb{Z}_{2}$. In the case where the magnetic
monopole vanishes and the Hamiltonian satisfies some asymptotic fibrewise
quadratic growth condition considered by Abbondandolo-Schwarz the
following remarkable result holds true.
\begin{thm}
\label{thm:[Abbondandolo-Schwarz-,-Salamon-}\emph{{[}Abbondandolo-Schwarz
\cite{AbbondandoloSchwarz2006}, Salamon-Weber \cite{SalamonWeber2006},
Viterbo \cite{Viterbo1996}{]}.} If the $\tau$-periodic $H$ satisfies
the Abbondandolo-Schwarz growth conditions, then Floer homology $HF_{*}^{\alpha}(H,0,\tau)$
is well-defined in every free homotopy class $\alpha$, and is isomorphic
to the singular homology of the space of $\tau$-periodic loops on
$M$ belonging to the given free homotopy class $\alpha$. 
\end{thm}
The precise definition of the Abbondandolo-Schwarz growth condition
is given in Definition \ref{AS growth conditions} below. In this
paper we will prove the following extension of Theorem \ref{thm:[Abbondandolo-Schwarz-,-Salamon-}.
\begin{namedthm}[\textbf{Theorem A}]
\textbf{}Assume that the $\tau$-periodic Hamiltonian $H$ satisfies
the Abbondandolo-Schwarz growth conditions and that $\sigma$ admits
a primitive of at most linear growth\textbf{\emph{ }}on the universal cover
of $M$. Then there exists $\delta_{0}(H,\sigma)>0$ such that if
$|\delta|\tau<\delta_{0}(H,\sigma)$ then the Floer homology $HF_{*}^{\alpha}(H,\delta\sigma,\tau)$
is well defined for all $\sigma$-atoroidal classes $\alpha\in[S^{1},M]$,
and is again isomorphic to the singular homology of the space of $\tau$-periodic
loops on $M$ belonging to the given free homotopy class $\alpha$.
If moreover $\sigma$ admits a bounded primitive on the universal
cover then $\delta_{0}(H,\sigma)=\infty$.
\end{namedthm}
We now explain the two new terms in the statement of Theorem A: a
primitive of at most linear growth, and $\sigma$-atoroidal free homotopy
class. 
\begin{defn*}
We say that $\sigma$ admits a \textbf{primitive with at most linear growth}
if the following condition holds:
\begin{lyxlist}{00.00.0000}
\item [{\textbf{(}$\boldsymbol{\sigma}$\textbf{1)}}] The $2$-form $\sigma$
is \textbf{weakly exact}. This means that the lift $\widetilde{\sigma}$
of $\sigma$ to $\widetilde{M}$ is exact (in particular, $\sigma$
is closed). Moreover $\widetilde{\sigma}$ admits a primitive with
\textbf{at most linear growth}: there exists $\theta\in\Omega^{1}(\widetilde{M})$
such that $d\theta=\widetilde{\sigma}$ and such that for any $z\in\widetilde{M}$
there exists a constant $\Theta_{z}>0$ such that for all $r\geq0$,
\begin{equation}
\sup_{q\in B(z,r)}\left|\theta_{q}\right|\leq\Theta_{z}(r+1).\label{eq:linear growth-1}
\end{equation}

\end{lyxlist}

Here $B(z,r)$ denotes the geodesic ball of radius $r$ in $\widetilde{M}$
about $z$, and both the geodesic metric and the norm $\left|\cdot\right|$
are defined using the lift of some Riemannian metric $g$ on $M$ to $\widetilde{M}$.
Asking whether $\widetilde{\sigma}$ has a primitive with at most linear
growth does not depend on the choice of metric $g$ on $M$. Moreover
as soon as \eqref{eq:linear growth-1} holds for some point $z\in\widetilde{M}$,
it holds for all $z\in\widetilde{M}$.

\end{defn*}
\noindent \textbf{Remarks. }
\begin{enumerate}
\item The condition \textbf{(}$\boldsymbol{\sigma}$\textbf{1) }includes
the following stronger condition: 

\begin{lyxlist}{00.00.0000}
\item [{\textbf{(}$\boldsymbol{\sigma}$\textbf{0)}}] The $2$-form $\sigma$
is \textbf{weakly exact}, and $\widetilde{\sigma}$ admits a \textbf{bounded
}primitive: there exists $\theta\in\Omega^{1}(\widetilde{M})$ such
that $d\theta=\widetilde{\sigma}$ and such that 
\begin{equation}
\sup_{q\in\widetilde{M}}\left|\theta_{q}\right|<\infty.\label{eq:bounded-1}
\end{equation}

\end{lyxlist}
\item A classical result of Gromov \cite{Gromov1991} tells us that if $M$
admits a metric of negative curvature then every closed 2-form $\sigma$
satisfies \textbf{(}$\boldsymbol{\sigma}$\textbf{0)}. In contrast,
if $\sigma$ is not exact and $\pi_{1}(M)$ is \textbf{amenable }then
$\sigma$ never satisfies \textbf{(}$\boldsymbol{\sigma}$\textbf{0)}
\cite[Corollary 5.4]{Paternain2006}. The main examples of pairs $(M,\sigma)$
where $\sigma$ satisfies \textbf{(}$\boldsymbol{\sigma}$\textbf{1)}
are given by manifolds that admit a metric of non-positive curvature
\cite{CaoXavier2001}. For tori $\mathbb{T}^{n}$ any closed non-exact 2-form
$\sigma$ satisfies \textbf{(}$\boldsymbol{\sigma}$\textbf{1)} but
not \textbf{(}$\boldsymbol{\sigma}$\textbf{0)}. 
\end{enumerate}
Set $\mathbb{S}_{\tau}:=\mathbb{R}/\tau\mathbb{Z}$. We often identify $S^1$ and $\mathbb{S}_1$. We denote by
$\Lambda_{\tau}M:=C^{\infty}(\mathbb{S}_{\tau},M)$ the free $\tau$-periodic
loop space of $M$. The space $\Lambda_{\tau}M$ splits as a direct
sum $\Lambda_{\tau}M=\bigoplus_{\alpha\in[S^{1},M]}\Lambda_{\tau}^{\alpha}M$,
where for a given free homotopy class $\alpha\in[S^{1},M]\cong[\mathbb{S}_{\tau},M]$,
we define $\Lambda_{\tau}^{\alpha}M:=\left\{ q\in\Lambda_{\tau}M\,:\,[q]=\alpha\right\} $.
Consider the 1-form $a_{\sigma}\in\Omega^{1}(\Lambda_{\tau}M)$ defined
by 
\begin{equation}
a_{\sigma}(q)(\xi):=\int_{\mathbb{S}_{\tau}}\sigma(\dot{q},\xi)dt,\label{eq:a sigma}
\end{equation}
Since $\sigma$ is closed, $a_{\sigma}$ is closed, that is, the integral of $a_{\sigma}$ over a closed path in $\Lambda_{\tau}M$ depends only the homology class of the path. 
\begin{defn*}
We say a class $\alpha\in[S^{1},M]$ is a \textbf{$\sigma$-atoroidal
class} if any map $f:S^{1}\rightarrow\Lambda_{1}^{\alpha}M$ with
$[f]=\alpha$ satisfies $\int_{S^{1}}f^{*}a_{\sigma}=0$. Equivalently,
$\alpha$ is a $\sigma$-atoroidal class if $a_{\sigma}|_{\Lambda_{\tau}^{\alpha}M}$
is exact (see \eqref{eq:why suwhy est}). Note that under the assumption that $\sigma$ is
weakly exact, the class $0$ of nullhomotopic loops is atoroidal,
since both statements are equivalent to the statement that $\sigma|_{\pi_{2}(M)}=0$.
\end{defn*}
Let us briefly comment how these assumptions enter the proof of Theorem
A:
\begin{enumerate}
\item The mere fact that $\sigma$ admits a primitive on the universal cover
implies that $\sigma$ vanishes on $\pi_{2}(M)$. Hence $\omega_{\sigma}$
is symplectically aspherical and no bubbling off of holomorphic spheres
can occur. This excludes the first obstruction to the compactness results needed 
to define the boundary operator. 
\item On $\sigma$-atoroidal classes the action functional used to define
the boundary operator is real valued, and hence the energy of its
gradient flow lines depends only on their asymptotes. This excludes
the second obstruction to the necessary compactness to define the
boundary operator. 
\item The third obstruction to compactness comes from the noncompactness
of $T^{*}M$. To obtain an $L^{\infty}$-bound on the perturbed holomorphic
curves we follow the approach by Abbondandolo and Schwarz
\cite{AbbondandoloSchwarz2006}. The assumption that $\sigma$ admits
a primitive with at most linear growth on the universal cover gives rise
to a certain quadratic isoperimetric inequality which allows us to
carry over the proof of Abbondandolo-Schwarz to this more general
set-up. This enables us to show that the Floer homology groups $HF_{*}^{\alpha}(H,\delta\sigma,\tau)$
are well defined. Taking advantage of the quadratic isoperimetric
inequality once more we construct a continuation isomorphism from
$HF_{*}^{\alpha}(H,\delta\sigma,\tau)$ to the Floer homology $HF_{*}^{\alpha}(H,0,\tau)$,
and hence Theorem \ref{thm:[Abbondandolo-Schwarz-,-Salamon-} implies
our result.
\end{enumerate}
The necessity of the assumption that $\left|\delta\right|\tau$ is
small in Theorem A can be seen from the following example. Take as
configuration space $M=\mathbb{T}^{2}$ the two-torus and as the Hamiltonian
$H$ take kinetic energy with respect to the standard flat metric
on the torus. As magnetic monopole we choose the area form $\sigma$
with respect to the standard metric, and work with period $\tau=1$.
Then for each $\delta\sigma$ the flow lines are either constant orbits
on $\mathbb{T}^{2}$ or lift to circles of period $2\pi/\delta$ on
the universal cover $\mathbb{R}^{2}$ of $\mathbb{T}^{2}$. Thus as
long as $\delta<2\pi$ the only periodic solutions of period one are
the constant ones. Hence the critical manifold is a two-torus and
the Floer homology is isomorphic to the homology of $\mathbb{T}^{2}$
which coincides with the homology of the contractible component of
the loop space of $\mathbb{T}^{2}$. However if $\delta=2\pi$ then
the critical manifold is diffeomorphic to $T^*\mathbb{T}^{2}$ and hence not compact anymore and one cannot define
Floer homology. If $\delta$ becomes larger than $2\pi$ the critical
manifold is again a two-torus. But one can check that the Conley-Zehnder
index of the critical manifold jumps by two once $\delta$ goes through
$2\pi$ and therefore the Floer homology now differs from the loop
space homology. 

However, on the torus it \textbf{is }in fact possible to define the
Floer homology $HF_{*}^{\alpha}(H,\delta\sigma,\tau)$ provided $\left|\delta\right|\tau\notin2\pi\mathbb{Z}$,
for \textbf{any }free homotopy class $\alpha\in[S^{1},\mathbb{T}^{2}]$.
More generally, let $g=\left\langle \cdot,\cdot\right\rangle$ denote a Riemannian metric on $\mathbb{T}^{2}$ and $f\in C^{\infty}(\mathbb{T}^{2},\mathbb{R})$.
Set $\sigma=f\,\mu_{g}$, and suppose there exists $k\in\mathbb{Z}$
such that 
\[
\frac{2\pi(k-1)}{\tau}<f(q)<\frac{2\pi k}{\tau}\ \ \ \mbox{for all }q\in\mathbb{T}^{2}.
\]
Fix $V\in C^{\infty}(\mathbb{S}_{\tau}\times \mathbb{T}^{2},\mathbb{R})$ and set
$H(t,q,p):=\frac{1}{2}\left|p\right|^{2}+V(t,q)$. Then $HF_{*}^{\alpha}(H,\sigma,\tau)$
is well defined for every free homotopy class $\alpha\in[S^{1},\mathbb{T}^{2}]$,
and moreover
\[
HF_{*}^{\alpha}(H,\sigma,\tau)=\begin{cases}
H_{*+2k}(\mathbb{T}^{2};\mathbb{Z}), & \alpha=0,\\
0, & \alpha\ne0.
\end{cases}
\]
The proof of this result is specific to tori, and as such goes along
somewhat different lines to that of Theorem A. For this reason the details of this proof will be discussed in a forthcoming paper.\newline

Theorem A has the following immediate corollary. Recall that if $H$ is given
by a Riemannian metric, the Hamiltonian flow of $X_{H,\sigma}$ is called
a {\bf magnetic flow}.

\begin{namedthm}[\textbf{Corollary B}]
\textbf{} Let $H$ be an autonomous Hamiltonian satisfying the Abbondandolo-Schwarz growth conditions and assume that
$\sigma$ admits a primitive with at most linear growth on the universal cover
of $M$. Let $\alpha\in [S^1,M]$ be a $\sigma$-atoroidal class.
Then for any $\tau >0$ sufficiently small, the Hamiltonian flow of $X_{H,\sigma}$ has a closed orbit with period $\tau$ whose projection to $M$ belongs to the class $\alpha$. In particular, the same is true for any magnetic flow associated to $\sigma$.
\end{namedthm}

To appreciate the significance of Corollary B consider the following example.
Let $M=\mathbb T^3$
and $\sigma$ any closed 2-form cohomologous to $dq_1\wedge dq_2$, where
$(q_1,q_2,q_3)$ are linear coordinates on the torus. It is easy to check that
the homotopy class $\alpha=(0,0,n)$ for any integer $n$ is $\sigma$-atoroidal.
Then given any metric on $\mathbb T^3$ the magnetic flow has a closed
orbit of period $\tau$ in the class $\alpha$ for all $\tau>0$ sufficiently small. In fact, for the standard flat metric and $\sigma=dq_1\wedge dq_2$, the
classes $(0,0,n)$ are the only ones that contain closed orbits of any period.\newline

Finally, in the case where $\widetilde{\sigma}$ admits a bounded
primitive, note that Theorem A tells us that in particular 
\[
HF_{n}^{0}(H,\sigma,\tau)\cong H_{n}(\Lambda_{\tau}^{0}M;\mathbb{Z}_{2})\ne0
\]
for \textbf{all} $\tau$-periodic Hamiltonians $H$ satisfying the
Abbondandolo-Schwarz growth conditions. As a consequence, Hein's proof
\cite{Hein2011} of the Conley conjecture for the cotangent bundle
(which is itself based on Ginzburg's proof \cite{Ginzburg2010} for
closed symplectically aspherical symplectic manifolds) goes through
word for word, and thus we obtain the following statement.

\begin{namedthm}[Corollary C]\emph{(The Conley Conjecture for twisted cotangent bundles)}
Assume that $\widetilde{\sigma}$ admits a bounded primitive. Let
$\varphi=\phi_{1}^{H}:T^{*}M\rightarrow T^{*}M$ denote the time-1
map of a Hamiltonian $H:S^1 \times T^{*}M\rightarrow\mathbb{R}$
satisfying the Abbondandolo-Schwarz growth conditions. Assume that
$\varphi$ has only finitely many fixed points. Then $\varphi$ has
simple periodic orbits of arbitrarily large period.
\end{namedthm}

This paper has an appendix in which we show that a more classical
approach is possible if we restrict to Hamiltonians that are in addition
strictly fibrewise convex. More precisely, we obtain a (Lagrangian)
action functional on the (completed) loop space $\Lambda_{\tau}^{\alpha}M$
which we show satisfies the Palais-Smale condition provided that $\sigma$
admits a primitive of at most linear growth, $H$ is fibrewise strictly
convex and satisfies the Abbondandolo-Schwarz growth conditions, and
$\left|\delta\right|\tau$ is sufficiently small. The Palais-Smale
condition allows the construction of the Morse complex so one can
recover again the homology of the loop space. We expect that it would
be possible to prove Corollary C in the Lagrangian setting by combining
the methods of the appendix with the work of Lu \cite{Lu2009,Lu2011} or Mazzucchelli \cite{Mazzucchelli2011a}.\newline

\emph{Acknowledgements}. We thank Viktor Ginzburg for pointing out to us that Corollary C followed
directly from Theorem A and the work of Hein. We also thank the referees for their careful reading of our manuscript, and for numerous helpful comments. U.\,Frauenfelder was
partially supported by the Basic Research grant 2010-0007669 funded
by the Korean government.

\section{Constructing the Floer homology $HF_{*}^{\alpha}(H,\sigma,\tau)$}

\subsection{\label{sub:Preliminaries}Preliminaries}

Denote by $\mathcal{R}(M)$ the set of Riemannian metrics on $M$.
Suppose $g=\left\langle \cdot,\cdot\right\rangle \in\mathcal{R}(M)$.
The metric defines a \textbf{horizontal-vertical} splitting of $TT^{*}M$:
given $z=(q,p)\in T^{*}M$
\[
T_{z}T^{*}M=T_{z}^{h}T^{*}M\oplus T_{z}^{v}T^{*}M\cong T_{q}M\oplus T_{q}^{*}M;
\]
here $T_{z}^{h}T^{*}M=\ker(\kappa_{g}:T_{z}T^{*}M\rightarrow T_{q}^{*}M)$,
where $\kappa_{g}$ is the connection map of the Levi-Civita connection
$\nabla$ of $g$, and $T_{z}^{v}T^{*}M=\ker(d\pi(z):T_{z}T^{*}M\rightarrow T_{q}M)$.
Given $\xi\in TT^{*}M$ we denote by $\xi^{h}$ and $\xi^{v}$ the
horizontal and vertical components. Technically speaking $\xi^{h}\in TM$
and $\xi^{v}\in T^{*}M$, although we consistently use the {}``musical''
isomorphism $v\mapsto\left\langle v,\cdot\right\rangle $ to identify
$TM$ with $T^{*}M$. The horizontal-vertical splitting also determines
an almost complex structure $J_{g}$ called the \textbf{metric almost
complex structure }via
\[
J_{g}=\left(\begin{array}{cc}
 & -\mathbb{1}\\
\mathbb{1}
\end{array}\right).
\]
Recall that an almost complex structure $J$ on $T^{*}M$ is $d\lambda$\textbf{-compatible}
if the bilinear form $G_{J}(\cdot,\cdot):=d\lambda(J\cdot,\cdot)$
defines a Riemannian metric on $T^{*}M$. The metric almost complex
structure $J_{g}$ is compatible for every Riemannian metric $g$
on $M$, and we abbreviate $G_{g}:=G_{J_{g}}$. We denote the set
of all $d\lambda$-compatible almost complex structures by $\mathcal{J}(T^{*}M)$ and equip it with the $C_{\textrm{loc}}^{\infty}$-topology.\newline 

Denote by $\Lambda_{\tau}T^{*}M:=C^{\infty}(\mathbb{S}_{\tau},T^{*}M)$.
Given $x=(q,p)\in\Lambda_{\tau}T^{*}M$ and $r\geq1$ we define 
\[
\left\Vert p\right\Vert _{L_{g}^{r}(\mathbb{S}_{\tau})}:=\left(\int_{0}^{\tau}\left|p\right|^{r}dt\right)^{1/r}\ \ \ \mbox{for }1\leq r<\infty,
\]
and
\[
\left\Vert p\right\Vert _{L_{g}^{\infty}(\mathbb{S}_{\tau})}:=\sup_{t\in\mathbb{S}_{\tau}}\left|p(t)\right|.
\]
Similarly given $\xi\in T_{x}\Lambda_{\tau}T^{*}M$ and $J\in\mathcal{J}(T^{*}M)$
we define 
\[
\left\Vert \xi\right\Vert _{L_{G_{J}}^{r}(\mathbb{S}_{\tau})}:=\left(\int_{0}^{\tau}\left[G_{J}(\xi,\xi)\right]^{r/2}dt\right)^{1/r}.
\]
Given $X\in\Gamma(\mbox{End}(TM))$ we define 
\[
\left\Vert X\right\Vert _{L_{g}^{\infty}}:=\sup_{q\in M}\sup\left\{ \left|X(q)v\right|\,:\, v\in T_{q}M,\ \left|v\right|=1\right\} .
\]
Let us now fix a closed 2-form $\sigma\in\Omega^{2}(M)$, and consider
the symplectic form $\omega_{\sigma}=d\lambda+\pi^{*}\sigma$ from
the Introduction. We denote by $\mathcal{J}_{\sigma}$ the open set
of almost complex structures $J$ on $T^{*}M$ that are \textbf{tamed
}by $\omega_{\sigma}$ - this just means that the bilinear form $\omega_{\sigma}(J\cdot,\cdot)$
is positive definite. We say that $J\in\mathcal{J}_{\sigma}$ is \textbf{uniformly
tame }if $J$ is also $d\lambda$-compatible (i.e. $J\in\mathcal{J}_{\sigma}\cap\mathcal{J}(T^{*}M)$),
and there exists some positive constant $\varepsilon>0$ such that
\[
\omega_{\sigma}(J\xi,\xi)\geq\varepsilon G_{J}(\xi,\xi)\ \ \ \mbox{for all }\xi\in TT^{*}M.
\]
The pair $(\sigma,g)$ defines a bundle endomorphism $Y=Y_{\sigma,g}\in\Gamma(\mbox{End}(TM))$
called the\textbf{ Lorentz force} of $\sigma$ via:
\[
\sigma_{q}(u,v)=\left\langle Y(q)u,v\right\rangle .
\]

The following lemma will be very useful.
\begin{lem}
\label{lem:uniformly tame-1}\textbf{\emph{(Uniformly tame almost
complex structures)}}
\begin{enumerate}
\item Fix $g\in\mathcal{R}(M)$. If 
\begin{equation}
\left\Vert Y_{\sigma,g}\right\Vert _{L_{g}^{\infty}}\leq1\label{eq:nice}
\end{equation}
 then the almost complex structure $J_{g}$ is uniformly tame (with
$\varepsilon=1/2$).
\item Denote by $\mathcal{R}_{\sigma}(M)\subseteq\mathcal{R}(M)$ the set
of Riemannian metrics on $M$ for which \eqref{eq:nice} holds. Given
any $g_{0}\in\mathcal{R}(M)$, if $\upsilon>\left\Vert Y_{\sigma,g}\right\Vert _{L_{g}^{\infty}}$
then the rescaled metric $g:=\upsilon g_{0}$ lies in $\mathcal{R}_{\sigma}(M)$.
\item Given $g\in\mathcal{R}(M)$ let 
\[
\mathcal{U}_{g}:=\left\{ J\in\mathcal{J}(T^{*}M)\,:\,\left\Vert J-J_{g}\right\Vert _{L_{G_{g}}^{\infty}}\leq1/7\right\} .
\]
Then if $g\in\mathcal{R}_{\sigma}(M)$ and $J\in\mathcal{U}_{g}$
then $J$ is uniformly tame (with $\varepsilon=1/4$): \textup{
\[
\omega_{\sigma}(J\xi,\xi)\geq\frac{1}{4}G_{J}(\xi,\xi)\ \ \ \mbox{for all }\xi\in TT^{*}M.
\]
}
\end{enumerate}
\end{lem}
\begin{proof}
(1). Write $Y=Y_{\sigma,g}$ and let $\xi\in TT^{*}M$. Then
\begin{align*}
\omega_{\sigma}(J_{g}\xi,\xi)-\frac{1}{2}G_{g}(\xi,\xi) & =\frac{1}{2}G_{g}(\xi,\xi)+\pi^{*}\sigma(J_{g}\xi,\xi)\\
 & =\frac{1}{2}G_{g}(\xi,\xi)+\left\langle Y(J_{g}\xi)^{h},\xi^{h}\right\rangle \\
 & =\frac{1}{2}G_{g}(\xi,\xi)-\left\langle Y\xi^{v},\xi^{h}\right\rangle ,\\
 & \geq\frac{1}{2}\left(\left|\xi^{h}\right|^{2}+\left|\xi^{v}\right|^{2}\right)-\left|\xi^{v}\right|\left|\xi^{h}\right|\\
 & =\left(\frac{1}{\sqrt{2}}\left|\xi^{h}\right|-\frac{1}{\sqrt{2}}\left|\xi^{v}\right|\right)^{2}\geq0.
\end{align*}

(2). For any $X\in\Gamma(\mbox{End}(TM))$ one has 
\[
\left\Vert X\right\Vert _{L_{\upsilon g}^{\infty}}=\left\Vert X\right\Vert _{L_{g}^{\infty}},
\]
and since $Y_{\sigma,\upsilon g}=\frac{1}{\upsilon}Y_{\sigma,g}$ we see
that if $\upsilon\geq\left\Vert Y_{\sigma,g}\right\Vert _{L_{g}^{\infty}}$
then $\left\Vert Y_{\sigma,\upsilon g}\right\Vert _{L_{\upsilon g}^{\infty}}\leq1$. 
(3). First note that for any $J\in\mathcal{J}(T^{*}M)$ we have 
\begin{align*}
G_{J}(\xi,\xi) & =\omega_{0}(J\xi,\xi)\\
 & =\omega_{0}(J_{g}J\xi,J_{g}\xi)\\
 & =G_{g}(J\xi,J_{g}\xi)\\
 & =G_{g}(J_{g}\xi,J_{g}\xi)+G_{g}((J-J_{g})\xi,J_{g}\xi)\\
 & \geq\left(1-\left\Vert J-J_{g}\right\Vert _{L_{G_{g}}^{\infty}}\right)G_{g}(\xi,\xi).
\end{align*}
Thus if $g\in\mathcal{R}_{\sigma}(M)$ and $J\in\mathcal{J}(T^{*}M)$
satisfies $\left\Vert J-J_{g}\right\Vert _{L_{G_{g}}^{\infty}}\leq1/7$
then 
\begin{align*}
\omega_{\sigma}(J\xi,\xi)-\frac{1}{4}G_{J}(\xi,\xi) & =\frac{3}{4}G_{J}(\xi,\xi)+\pi^{*}\sigma(J\xi,\xi)\\
 & =\frac{3}{4}G_{J}(\xi,\xi)+\left\langle Y(J_{g}\xi)^{h},\xi^{h}\right\rangle +\left\langle Y((J-J_{g})\xi)^{h},\xi^{h}\right\rangle \\
 & \overset{(*)}{\geq}\frac{3}{4}G_{J}(\xi,\xi)-\frac{1}{2}G_{g}(\xi,\xi)-\left\Vert J-J_{g}\right\Vert _{L_{G_{g}}^{\infty}}G_{g}(\xi,\xi)\\
 & \geq\frac{1}{4}G_{g}(\xi,\xi)-\frac{3}{4}\left\Vert J-J_{g}\right\Vert _{L_{G_{g}}^{\infty}}G_{g}(\xi,\xi)-\left\Vert J-J_{g}\right\Vert _{L_{G_{g}}^{\infty}}G_{g}(\xi,\xi)\\
 & \geq\frac{1}{4}\left(1-7\left\Vert J-J_{g}\right\Vert _{L_{G_{g}}^{\infty}}\right)G_{g}(\xi,\xi)\geq0,
\end{align*}
where $(*)$ used the first part of the lemma.

\end{proof}
Assume that $\alpha\in[S^{1},M]$ is a $\sigma$-atoroidal class.
Fix a reference point $*\in M$, and fix a reference loop $q_{\alpha}\in\Lambda_{1}^{\alpha}M$
such that $q_{\alpha}(0)=*$. Given any $q\in\Lambda_{\tau}^{\alpha}M$,
we define 
\[
\mathcal{A}_{\sigma}(q):=\int_{[0,1]\times\mathbb{S}_{\tau}}w^{*}\sigma,
\]
where $w:[0,1]\times\mathbb{S}_{\tau}\rightarrow M$ is any smooth
map such that $w(0,t)=q_{\alpha}(t/\tau)$ and $w(1,t)=q(t)$. Since
$\alpha$ is $\sigma$-atoroidal, $\mathcal{A}_{\sigma}$ is well
defined (i.e. independent of the choice of $w$), and one sees immediately
that 
\begin{equation}
d\mathcal{A}_{\sigma}=a_{\sigma}\ \ \ \mbox{on }\Lambda_{\tau}^{\alpha}M.\label{eq:why suwhy est}
\end{equation}
Fix a point $\widetilde{*}\in\widetilde{M}$ that projects onto our
fixed reference point $*\in M$. We denote by $\widetilde{q}_{\alpha}:[0,1]\rightarrow\widetilde{M}$
the lift of $q_{\alpha}$ to $\widetilde{M}$ with $\widetilde{q}_{\alpha}(0)=\widetilde{*}$.
The following lemma is based on \cite[Lemma 2.4]{BaeFrauenfelder2010},
and explains the importance of the condition \textbf{($\boldsymbol{\sigma}$1)}.
\begin{lem}
\label{lem:BF-2}\textbf{\emph{(The quadratic isoperimetric inequality)}}

Assume $\sigma$ satisfies condition \textbf{\emph{($\boldsymbol{\sigma}$1)}}
and $\alpha\in[S^{1},M]$ is a $\sigma$-atoroidal class. There exists
a constant $C_{0}=C_{0}(\sigma,g)>0$ and a constant $C_{1}=C_{1}(\sigma,g,\alpha)>0$
such that for all $q\in\Lambda_{\tau}^{\alpha}M$ one has 
\[
\left|\mathcal{A}_{\sigma}(q)\right|\leq C_{0}\left(\int_{0}^{\tau}\left|\dot{q}(t)\right|dt\right)^{2}+C_{1}.
\]
\end{lem}
\begin{proof}
Let $\theta$ denote a primitive of $\widetilde{\sigma}$ such that
for all $z\in\widetilde{M}$ there exists a constant $\Theta_{z}$
such that 
\[
\sup_{q\in B(z,r)}\left|\theta_{q}\right|\leq\Theta_{z}(r+1).
\]
Now let $q\in\Lambda_{\tau}^{\alpha}M$, and let $w:[0,1]\times\mathbb{S}_{\tau}\rightarrow M$
denote a smooth map such that $w(0,t)=q_{\alpha}(t/\tau)$ and $w(1,t)=q(t)$,
together with the additional property that if $\widetilde{w}:[0,1]\times[0,\tau]\rightarrow\widetilde{M}$
denotes the lifting of $w$ to the universal cover such that $\widetilde{w}(0,t)=\widetilde{q}_{\alpha}(t/\tau)$
then 
\[
\int_{0}^{1}\left|\partial_{s}\widetilde{w}(s,i)\right|ds\leq d,\ \ \ \mbox{for }i=0,1,
\]
where $d:=\mbox{diam}(M,g)$. Set
\[
\Theta:=\Theta_{\widetilde{*}},\ \ \ \ell_{\alpha}:=\int_{0}^{1}\left|\dot{q}_{\alpha}(t)\right|dt,\ \ \ \ell(q):=\int_{0}^{\tau}\left|\dot{q}(t)\right|dt.
\]
Then we have
\begin{align*}
\left|\mathcal{A}_{\sigma}(q)\right| & =\left|\int_{[0,1]\times\mathbb{S}_{\tau}}w^{*}\sigma\right|\\
 & =\left|\int_{[0,1]\times[0,\tau]}\widetilde{w}^{*}\widetilde{\sigma}\right|\\
 & \leq\left|\int_{0}^{1}\theta(\partial_{s}\widetilde{w}(s,0))ds\right|+\left|\int_{0}^{\tau}\theta(\partial_{t}\widetilde{w}(1,t))dt\right|+\left|\int_{0}^{1}\theta(\partial_{s}\widetilde{w}(s,1))ds\right|+\left|\int_{0}^{\tau}\theta(\partial_{t}\widetilde{w}(0,t))dt\right|\\
 & \leq\Theta(d+1)d+\Theta(d+\ell(q)+1)\ell(q)+\Theta(d+\ell_{\alpha}+1)d+\Theta(\ell_{\alpha}+1)\ell_{\alpha}.
\end{align*}
The desired statement follows with 
\[
C_{0}:=(2+d)\Theta;
\]
\[
C_{1}:=\Theta(d+1)d+\Theta(d+1)+\Theta(d+\ell_{\alpha}+1)d+\Theta(\ell_{\alpha}+1)\ell_{\alpha}.
\]

\end{proof}

\subsection{The action functional}

Throughout this section assume that $\sigma$ satisfies \textbf{($\boldsymbol{\sigma}$1)}
and $\alpha\in[S^{1},M]$ is a $\sigma$-atoroidal class. 

Fix a $\tau$-periodic Hamiltonian $H:\mathbb{S}_{\tau}\times T^{*}M\rightarrow\mathbb{R}$.
Denote by $\mathcal{P}_{\tau}^{\alpha}(H,\sigma)\subseteq\Lambda_{\tau}^{\alpha}T^{*}M$
the set of closed $\tau$-periodic orbits of $X_{H,\sigma}$ belonging
to $\Lambda_{\tau}^{\alpha}T^{*}M$:
\[
\mathcal{P}_{\tau}^{\alpha}(H,\sigma)=\left\{ x\in\Lambda_{\tau}^{\alpha}T^{*}M\,:\,\dot{x}=X_{H,\sigma}(t,x)\right\} 
\]
($\Lambda_{\tau}^{\alpha}T^{*}M$ denotes those $\tau$-periodic loops
$x$ whose projection to $M$ lies in $\Lambda_{\tau}^{\alpha}M$).
Denote by $\phi_{t}^{H,\sigma}:T^{*}M\rightarrow T^{*}M$ the flow
of $X_{H,\sigma}$. In order to construct the Floer complex associated
to $H,\sigma$ and $\alpha$ we need to make the following standard
assumption on the triple $(H,\sigma,\alpha)$: 
\begin{lyxlist}{00.00.0000}
\item [{\textbf{(N)}}] All the elements $x\in\mathcal{P}_{\tau}^{\alpha}(H,\sigma)$
are \textbf{non-degenerate}, that is, the linear map $d\phi_{\tau}^{H,\sigma}(x(0))\in\mathsf{Sp}(T_{x(0)}T^{*}M)$
does not have 1 as an eigenvalue. \end{lyxlist}
\begin{rem}
In fact, as far as Theorem A is concerned, the assumption that Condition
\textbf{(N) }is satisfied can be relaxed. Indeed, the point is that
by making a very small perturbation of $H$ along the 1-periodic orbits
we can create a new Hamiltonian $\widetilde{H}$ which still satisfies\textbf{
}all the other requirements of Theorem A, and also such that $\widetilde{H}$
satisfies Condition \textbf{(N)}. Then Theorem A tells us that the
Floer homology $HF_{*}^{\alpha}(\widetilde{H},\delta\sigma,\tau)$
is well defined, and moreover if $\widehat{H}$ is another such perturbation
then by Theorem \ref{thm:third statement} below we have $HF_{*}^{\alpha}(\widetilde{H},\sigma,\tau)\cong HF_{*}^{\alpha}(\widehat{H},\sigma,\tau)$.
In other words, we can still define $HF_{*}^{\alpha}(H,\sigma,\tau)$
even when Condition \textbf{(N) }is not satisfied, by simply setting
\[
HF_{*}^{\alpha}(H,\sigma,\tau)\overset{\textrm{def}}{=}HF_{*}^{\alpha}(\widetilde{H},\sigma,\tau)
\]
 for any such perturbation $\widetilde{H}$. 
\end{rem}
The action functional $\mathcal{A}_{H,\sigma}:\Lambda_{\tau}^{\alpha}T^{*}M\rightarrow\mathbb{R}$
that we will work with is defined by
\[
\mathcal{A}_{H,\sigma}(x):=\mathcal{A}_{H}(x)+\mathcal{A}_{\sigma}(\pi\circ x),
\]
where $\mathcal{A}_{H}$ denotes the \textbf{standard} \textbf{Hamiltonian
action functional}
\[
\mathcal{A}_{H}(x):=\int_{\mathbb{S}_{\tau}}\lambda^{*}x-\int_{0}^{\tau}H(t,x)dt.
\]
It is not hard to check that a loop $x\in\Lambda_{\tau}^{\alpha}T^{*}M$
is a critical point of $\mathcal{A}_{H,\sigma}$ if and only if $x\in\mathcal{P}_{\tau}^{\alpha}(H,\sigma)$.

In order to be able to obtain the necessary compactness results needed
to define the Floer homology, following Abbondandolo and Schwarz \cite[Section 1.5]{AbbondandoloSchwarz2006}
we impose two growth conditions on $H$. In the statement of the following
definition, $Z\in\mbox{Vect}(T^{*}M)$ denotes the \textbf{Liouville
vector field}, which is uniquely defined by the equation $i_{Z}d\lambda=\lambda$. 
\begin{defn}
\label{AS growth conditions}A Hamiltonian $H\in C^{\infty}(\mathbb{S}_{\tau}\times T^{*}M,\mathbb{R})$
satisfies the \textbf{Abbondandolo-Schwarz growth conditions }if the
following two requirements hold:
\begin{lyxlist}{00.00.0000}
\item [{\textbf{(H1)}}] There exists $h_{1}>0$ and $k_{1}\geq0$ such
that 
\[
dH(t,q,p)Z(q,p)-H(t,q,p)\geq h_{1}\left|p\right|^{2}-k_{1}\ \ \ \mbox{for all }(t,q,p)\in\mathbb{S}_{\tau}\times T^{*}M.
\]

\item [{\textbf{(H2)}}] There exists $h_{2}>0$ and $k_{2}\geq0$ such
that 
\[
\left|\nabla_{q}H(t,q,p)\right|\leq h_{2}\left|p\right|^{2}+k_{2}\ \ \ \mbox{for all }(t,q,p)\in\mathbb{S}_{\tau}\times T^{*}M,
\]
\begin{equation}
\left|\nabla_{p}H(t,q,p)\right|^{2}\leq h_{2}\left|p\right|^{2}+k_{2}\ \ \ \mbox{for all }(t,q,p)\in\mathbb{S}_{\tau}\times T^{*}M.\label{eq:p part of H2}
\end{equation}

\end{lyxlist}

Here we have chosen a Riemannian metric $g$ on $M$, and $\nabla_{q}H$
and $\nabla_{p}H$ denote the horizontal and vertical components of
the gradient $\nabla H$ of $H$ under the splitting $TT^{*}M\cong TM\oplus T^{*}M$
induced by the Riemannian metric (see Section \ref{sub:The-Lagrangian-action}
for the precise definition). Whilst the constants $h_{i},k_{i}$ depend
on the choice of metric $g$ on $M$, the existence of such constants
does not (see \cite[p273]{AbbondandoloSchwarz2006}). 

\end{defn}
We now define a constant $\delta_{0}(H,\sigma)$ associated to a pair
$(H,\sigma)$, where $H$ satisfies the Abbondandolo-Schwarz growth
conditions, and $\sigma$ satisfies \textbf{($\boldsymbol{\sigma}$1)}.
This is the constant that appears in the statement of Theorem A.
\begin{defn}
Firstly, given a Riemannian metric $g\in\mathcal{R}(M)$, define 
\[
\eta_{1}(H,g):=\sup\left\{ h_{1}>0\,:\, H\ \mbox{satisfies \textbf{(H1) }with respect to }h_{1}\mbox{ and some }k_{1}\geq0\right\} ;
\]
\[
\eta_{2}(H,g):=\inf\left\{ h_{2}>0\,:\, H\mbox{ satisfies }\eqref{eq:p part of H2}\mbox{ with respect to }h_{2}\mbox{ and some }k_{2}\geq0\right\} .
\]
The reason that $\eta_{2}(H,g)$ is the infimum over the constants
$h_{2}$ for which \eqref{eq:p part of H2} is satisfied (rather than
over the constants $h_{2}$\textbf{ }for which \textbf{(H2) }is satisfied)
is that this part of the argument - specifically, Lemma \ref{lem:claim 2}
- does not require any assumptions on the growth of $\nabla_{q}H$.
This assumption comes into play later on, cf. Section \ref{sec:Proofs}.
Note that if $H$ satisfies both \textbf{(H1) }and \textbf{(H2) }then
$0<\eta_{1}(H,g),\eta_2(H,g)<\infty$. 

Now set
\begin{equation}
\delta_{0}(H,\sigma,g):=\begin{cases}
\frac{\eta_{1}(H,g)}{2C_{0}(\sigma,g)\eta_{2}(H,g)}, & \mbox{if }\sigma\mbox{ satisfies \textbf{(\ensuremath{\boldsymbol{\sigma}}1)} but not \textbf{(\ensuremath{\boldsymbol{\sigma}}0)},}\\
\infty, & \mbox{if }\sigma\mbox{ satisfies \textbf{(\ensuremath{\boldsymbol{\sigma}}0)}},
\end{cases}\label{eq:delta0}
\end{equation}
where the constant $C_0(\sigma,g)$ was defined in Lemma \ref{lem:BF-2}.
Finally set 
\[
\delta_{0}(H,\sigma):=\sup_{g\in\mathcal{R}(M)}\delta_{0}(H,\sigma,g)\in(0,\infty].
\]
\end{defn}
\begin{rem}
\label{can assume g is in RsigmaM}Observe that 
\[
\delta_{0}(H,\sigma,\upsilon g)=\delta_{0}(H,\sigma,g)\ \ \ \mbox{for all }\upsilon>0.
\]
Thus we can alternatively define
\[
\delta_{0}(H,\sigma):=\sup_{g\in\mathcal{R}_{\sigma}(M)}\delta_{0}(H,\sigma,g)\in(0,\infty],
\]
where $\mathcal{R}_{\sigma}(M)$ was defined in Lemma \ref{lem:uniformly tame-1}.2.
\end{rem}

\subsection{The Floer equation}

Fix a Riemannian metric $g\in\mathcal{R}_{\sigma}(M)$, and a Hamiltonian
$H\in C^{\infty}(\mathbb{S}_{\tau}\times T^{*}M,\mathbb{\mathbb{R}})$
satisfying the Abbondandolo-Schwarz growth conditions \textbf{(H1)
}and \textbf{(H2)}. Condition \textbf{(H2) }implies that there exists
a constant $h_{\sigma,g}\geq0$ such that 
\begin{equation}
\left|X_{H,\sigma}(t,q,p)\right|\leq h_{\sigma,g}\left(1+\left|p\right|^{2}\right)\ \ \ \mbox{for all }(t,q,p)\in\mathbb{S}_{\tau}\times T^{*}M.\label{eq:the condition on X}
\end{equation}
Observe that

\begin{equation}
\lambda(X_{H,\sigma})=dH(Z)\label{eq:dh z}
\end{equation}
(recall $Z$ denotes the Liouville vector field); in particular $\lambda(X_{H,\sigma})$ does \textbf{not }depend on
$\sigma$. Indeed, 
\[
\lambda(X_{H,\sigma})+\pi^{*}\sigma(Z,X_{H,\sigma})=\omega_{\sigma}(Z,X_{H,\sigma})=dH(Z),
\]
and 
\[
\pi^{*}\sigma(Z,X_{H,\sigma})=0
\]
as $d\pi(Z)\equiv0$.

Fix a $\sigma$-atoroidal class $\alpha\in[S^{1},M]$ and $\delta\in\mathbb{R}$.
Thus the action functional $\mathcal{A}_{H,\delta\sigma}:\Lambda_{\tau}^{\alpha}T^{*}M\rightarrow\mathbb{R}$
is defined. Given a family $\mathbf{J}=(J_{t})_{t\in\mathbb{S}_{\tau}}\subseteq\mathcal{J}_{\sigma}$,
denote by $\nabla_{\mathbf{J}}\mathcal{A}_{H,\delta\sigma}$ the vector
field on $\Lambda_{\tau}^{\alpha}T^{*}M$ defined by 
\[
\nabla_{\mathbf{J}}\mathcal{A}_{H,\delta\sigma}(x)=J_{t}(x)(\dot{x}-X_{H,\delta\sigma}(t,x)).
\]
With these definitions one has
\begin{equation}
d\mathcal{A}_{H,\sigma}(x)(\xi)=\left\langle \left\langle \nabla_{\mathbf{J}}\mathcal{A}_{H,\sigma}(x),\xi\right\rangle \right\rangle _{L_{G_{\mathbf{J}}}^{2}(\mathbb{S}_{\tau})},\label{eq:the gradient}
\end{equation}
where $\left\langle \left\langle \cdot,\cdot\right\rangle \right\rangle _{L_{G_{\mathbf{J}}}^{2}(\mathbb{S}_{\tau})}$
denotes the possibly non-symmetric inner product given by 
\[
\left\langle \left\langle \xi,\zeta\right\rangle \right\rangle _{L_{G_{\mathbf{J}}}^{2}(\mathbb{S}_{\tau})}:=\int_{0}^{\tau}\omega_{\sigma}(J_{t}\xi,\zeta)dt.
\]
 We remind the reader that since the almost complex structures $\mathbf{J}=(J_{t})_{t\in\mathbb{S}_{\tau}}$
are only assumed to be tamed by $\omega_{\sigma}$ (rather than compatible),
the order in \eqref{eq:the gradient} is important, that is, in general
\[
d\mathcal{A}_{H,\sigma}(x)(\xi)\ne\left\langle \left\langle \xi,\nabla_{\mathbf{J}}\mathcal{A}_{H,\sigma}(x)\right\rangle \right\rangle _{L_{G_{\mathbf{J}}}^{2}(\mathbb{S}_{\tau})}.
\]

Given critical points $x_{-},x_{+}\in\mathcal{P}_{\tau}^{\alpha}(H,\delta\sigma)$
we denote by 
\[
\mathcal{M}_{\tau}^{\alpha}(x_{-},x_{+},H,\delta\sigma,\mathbf{J})\subseteq C^{\infty}(\mathbb{R}\times\mathbb{S}_{\tau},T^{*}M)
\]
the set of smooth maps $u:\mathbb{R}\times\mathbb{S}_{\tau}\rightarrow T^{*}M$
that satisfy the \textbf{Floer equation}
\begin{equation}
\partial_{s}u+\nabla_{\mathbf{J}}\mathcal{A}_{H,\delta\sigma}(u)=0\label{eq:floer}
\end{equation}
and submit to the asymptotic conditions
\begin{equation}
\lim_{s\rightarrow\pm\infty}u(s,t)=x_{\pm}(t),\ \ \ \lim_{s\rightarrow\pm\infty}\partial_{s}u(s,t)=0,\label{eq:asymptotes}
\end{equation}
both limits being uniform in $t$. 

More generally, we denote by $\mathcal{M}_{\tau}^{\alpha}(a,b,H,\delta\sigma,\mathbf{J})$
the set of maps $u\in C^{\infty}(\mathbb{R}\times\mathbb{S}_{\tau},T^{*}M)$
satisfying \eqref{eq:floer} and 
\[
a\leq\mathcal{A}_{H,\delta\sigma}(u(s,t))\leq b\ \ \ \mbox{for all }(s,t)\in\mathbb{R}\times\mathbb{S}_{\tau}.
\]

Recall the definition of the the set $\mathcal{U}_{g}$ of almost
complex structures from Lemma \ref{lem:uniformly tame-1}.3. The following
theorem is central to defining the Floer homology $HF_{*}^{\alpha}(H,\tau,\sigma)$,
and will be proved in Section \ref{sec:Proofs} below.
\begin{thm}
\label{thm:AS1}\textbf{\emph{($L^{\infty}$ bounds on gradient flow
lines)}}

Suppose $g\in\mathcal{R}_{\sigma}(M)$, $\tau\left|\delta\right|<\delta_{0}(H,\sigma,g)$
and $\alpha\in[S^{1},M]$ is a $\sigma$-atoroidal class. There exists
a smaller neighborhood $\mathcal{V}_{g}\subseteq\mathcal{U}_{g}$
of $J_{g}$ such that for any family $\mathbf{J}=(J_{t})_{t\in\mathbb{S}_{\tau}}\subseteq\mathcal{V}_{g}$,
and for all $-\infty<a\leq b<\infty$, there exists a compact set
$K=K(a,b,\mathbf{J})\subseteq T^{*}M$ such that for any $u\in\mathcal{M}_{\tau}^{\alpha}(a,b,H,\delta\sigma,\mathbf{J})$
one has 
\[
u(\mathbb{R}\times\mathbb{S}_{\tau})\subseteq K.
\]

\end{thm}

\subsection{Defining the Floer homology groups}

Let us now fix:
\begin{itemize}
\item a closed 2-form $\sigma\in\Omega^{2}(M)$ that satisfies Condition
\textbf{($\boldsymbol{\sigma}$1)}, 
\item a $\sigma$-atoroidal class $\alpha\in[S^{1},M]$,
\item a Riemannian metric $g\in\mathcal{R}_{\sigma}(M)$,
\item a Hamiltonian $H\in C^{\infty}(\mathbb{S}_{\tau}\times T^{*}M,\mathbb{R})$
that satisfies the Abbondandolo-Schwarz growth conditions \textbf{(H1)
}and \textbf{(H2)},
\item a constant $\delta\in\mathbb{R}$ such that $\tau\left|\delta\right|<\delta_{0}(H,\sigma,g)$,
and such that $(H,\delta\sigma,\alpha)$ satisfies Condition \textbf{(N)}. 
\end{itemize}
We will now explain how Theorem \ref{thm:AS1} allows us to define
the Floer homology groups $HF_{*}^{\alpha}(H,\delta\sigma,\tau)$.
All of this material is now standard (and essentially identical to
\cite[Section 1.7]{AbbondandoloSchwarz2006}), and we refer the reader
to any of a number standard sources (e.g. Salamon's lecture notes
\cite{Salamon1999}) for more details.

For each $x\in\mathcal{P}_{\tau}^{\alpha}(H,\delta\sigma)$, let $\mu_{\textrm{CZ}}(x)$
denote the \textbf{Conley-Zehnder index }of $x$. In order to define
the Conley-Zehnder index we choose a vertical preserving symplectic
trivialization (see \cite{AbbondandoloSchwarz2006}); the fact that
$c_{1}(T^{*}M,\omega_{\sigma})=0$ means that the value of $\mu_{\textrm{CZ}}(x)$
is independent of this choice of trivialization. Note however that
our sign conventions match those of \cite{AbbondandoloSchwarz2010}
not \cite{AbbondandoloSchwarz2006}. The non-degeneracy condition
\textbf{(N) }implies that $\mu_{\textrm{CZ}}(x)$ is always an integer. 

Given $k\in\mathbb{Z}$ let 
\[
\mathcal{P}_{\tau}^{\alpha}(H,\delta\sigma)_{k}:=\{x\in\mathcal{P}_{\tau}^{\alpha}(H,\delta\sigma)\,:\,\mu_{\textrm{CZ}}(x)=k\}.
\]
The moduli spaces $\mathcal{M}_{\tau}^{\alpha}(x_{-},x_{+},H,\delta\sigma,\mathbf{J})$
all carry a free $\mathbb{R}$-action given by $(s_{0}\cdot u)(s,t):=u(s-s_{0},t)$,
and we denote by $\mathcal{M}_{\tau}^{\alpha}(x_{-},x_{+},H,\delta\sigma,\mathbf{J})/\mathbb{R}$
the quotient space under this action. For a generic choice of\textbf{
}$\mathbf{J}=(J_{t})_{t\in\mathbb{S}_{\tau}}\subseteq\mathcal{V}_{g}$,
it follows from Theorem \ref{thm:AS1} and standard Floer-theoretic
arguments that the quotient moduli spaces $\mathcal{M}_{\tau}^{\alpha}(x_{-},x_{+},H,\delta\sigma,\mathbf{J})/\mathbb{R}$
all carry the structure of a $(\mu_{\textrm{CZ}}(x_{-})-\mu_{\textrm{CZ}}(x_{+})-1)$-dimensional
manifold. Moreover if $\mu_{\textrm{CZ}}(x_{-})=\mu_{\textrm{CZ}}(x_{+})+1$
then $\mathcal{M}_{\tau}^{\alpha}(x_{-},x_{+},H,\delta\sigma,\mathbf{J})/\mathbb{R}$
is actually compact (and hence a finite set).

We define the \textbf{Floer chain group} $CF_{k}^{\alpha}(H,\delta\sigma,\tau)$
to be the free $\mathbb{Z}_{2}$-module generated by the elements
of $\mathcal{P}_{\tau}^{\alpha}(H,\delta\sigma)_{k}$. Note that $CF_{\tau}^{\alpha}(H,\delta\sigma,\tau)$
may not be finitely generated. The boundary operator $\partial(\mathbf{J}):CF_{k}^{\alpha}(H,\delta\sigma,\tau)\rightarrow CF_{k-1}^{\alpha}(H,\delta\sigma,\tau)$
is defined by
\[
\partial(\mathbf{J})(x):=\sum_{y\in\mathcal{P}_{\tau}^{\alpha}(H,\delta\sigma)_{k-1}}n(x,y)y,\ \ \ x\in\mathcal{P}_{\tau}^{\alpha}(H,\delta\sigma)_{k},
\]
where
\[
n(x,y):=\#_{2}\left(\mathcal{M}_{\tau}^{\alpha}(x,y,H,\delta\sigma,\mathbf{J})/\mathbb{R}\right)
\]
denotes the parity of the finite set $\mathcal{M}_{\tau}^{\alpha}(x,y,H,\delta\sigma,\mathbf{J})/\mathbb{R}$.
This is well defined since the sum contains only finitely many non-zero
terms, thanks to the forthcoming Remark \ref{finite critical pts}.

The usual argument \cite{Salamon1999}, tells us that $\partial(\mathbf{J})\circ\partial(\mathbf{J})=0$,
and hence we may define the\textbf{ Floer homology }$HF_{*}^{\alpha}(H,\delta\sigma,\tau)$
to be the homology of the chain complex $\{CF_{*}^{\alpha}(H,\delta\sigma,\tau),\partial(\mathbf{J})\}$.
It is acceptable to omit the $\mathbf{J}$ from the notation for the
homology $HF_{*}^{\alpha}(H,\delta\sigma,\tau)$, as any two (generically
chosen) families $\mathbf{J}$ and $\mathbf{J}'$ produce chain homotopic
chain complexes (see \cite[Theorem 1.20]{AbbondandoloSchwarz2006}).

\section{Proofs}

\subsection{\label{sec:Proofs}The proof of Theorem \ref{thm:AS1}}

As mentioned in the Introduction, our proof of Theorem \ref{thm:AS1}
will closely follow Abbondandolo and Schwarz' method in \cite{AbbondandoloSchwarz2006}.
Their method has two distinct stages. The first stage appears as Lemma
1.12 in \cite{AbbondandoloSchwarz2006}, and asserts that under the
hypotheses of the theorem, $ $there exists a constant $R=R(a,b)>0$
such that for any $u=(q,p)\in\mathcal{M}_{\tau}^{\alpha}(a,b,H,\delta\sigma,\mathbf{J})$
and any interval $I\subseteq\mathbb{R}$ it holds that 
\begin{equation}
\left\Vert p\right\Vert _{W_{g}^{1,2}(I\times\mathbb{S}_{\tau})}\leq R\left(\left|I\right|^{1/2}+1\right).\label{eq:first stage}
\end{equation}
This stage uses heavily the fact that $H$ satisfies conditions \textbf{(H1)
}and \textbf{(H2)}.
The second stage appears as Theorem 1.14 in \cite{AbbondandoloSchwarz2006}.
Roughly speaking, the second stage works as follows: firstly, by Nash's
Theorem, we may isometrically embed the Riemannian manifold $(M,g)$
into $(\mathbb{R}^{N},g_{\textrm{eucl}})$. This embedding in turn
induces an isometric embedding of $(TT^{*}M,G_{g})$ into $(\mathbb{R}^{2N},g_{\textrm{eucl}})$.
Under this embedding if $\mathtt{i}$ denotes the canonical almost
complex structure on $\mathbb{R}^{2N}$ given by 
\[
\mathtt{i}=\left(\begin{array}{cc}
 & -\mathbb{1}\\
\mathbb{1}
\end{array}\right)
\]
then $\mathtt{i}|_{T^{*}M}=J_{g}$. The proof then uses Calderon-Zygmund
estimates for the Cauchy-Riemann operator, together with certain interpolation
inequalities, to upgrade equation \eqref{eq:first stage} to the full
statement of Theorem \ref{thm:AS1}. These estimates only work for
$\mathbf{J}$ contained in a sufficiently small neighborhood $\mathcal{W}_{g}$
of $J_{g}$: the set $\mathcal{V}_{g}$ in the statement of Theorem
\ref{thm:AS1} is then defined by $\mathcal{V}_{g}:=\mathcal{U}_{g}\cap\mathcal{W}_{g}$.
The proof of this stage goes through word for word in our situation,
and thus in order to prove Theorem \ref{thm:AS1} it suffices to prove
the first stage, namely equation \eqref{eq:first stage}.

The proof of \eqref{eq:first stage} (Lemma 1.12 in \cite{AbbondandoloSchwarz2006})
consists of six claims. A careful inspection of their proof shows
that everything apart from Claim 1 and Claim 2 goes through verbatim
in our case. Claims 1 and 2 however require a little more work. The
following lemma proves Claim 1.
\begin{lem}
\label{lem:claim 1}Fix $g\in\mathcal{R}_{\sigma}(M)$. If $\mathbf{J}=(J_{t})_{t\in\mathbb{S}_{\tau}}\subseteq\mathcal{U}_{g}$
and $u:\mathbb{R}\times\mathbb{S}_{\tau}\rightarrow T^{*}M$ satisfies
\eqref{eq:floer} and \eqref{eq:asymptotes} with respect to $\mathbf{J}$,
then 
\[
\left\Vert \partial_{s}u\right\Vert _{L_{G_{g}}^{2}(\mathbb{R}\times\mathbb{S}_{\tau})}^{2}\leq4\sup_{t\in\mathbb{S}_{\tau}}\left\Vert J_{t}\right\Vert _{L_{G_{g}}^{\infty}}^{2}(\mathcal{A}_{H,\sigma}(x_{-})-\mathcal{A}_{H,\sigma}(x_{+})).
\]
\end{lem}
\begin{proof}
The proof is a simple computation using  Lemma \ref{lem:uniformly tame-1}.3 and \eqref{eq:the gradient}.

\begin{align*}
\left\Vert \partial_{s}u\right\Vert _{L_{G_{g}}^{2}(\mathbb{R}\times\mathbb{S}_{\tau})}^{2} & \leq\sup_{t\in\mathbb{S}_{\tau}}\left\Vert J_{t}\right\Vert _{L_{G_{g}}^{\infty}}^{2}\left\Vert \partial_{s}u\right\Vert _{L_{G_{J_{t}}}^{2}(\mathbb{R}\times\mathbb{S}_{\tau})}^{2}\\
 & \leq4\sup_{t\in\mathbb{S}_{\tau}}\left\Vert J_{t}\right\Vert _{L_{G_{g}}^{\infty}}^{2}\int_{-\infty}^{\infty}\int_{0}^{\tau}\omega_{\sigma}(J_{t}\partial_{s}u,\partial_{s}u)dtds\\
 & =4\sup_{t\in\mathbb{S}_{\tau}}\left\Vert J_{t}\right\Vert _{L_{G_{g}}^{\infty}}^{2}\int_{-\infty}^{\infty}(-d\mathcal{A}_{H,\sigma}(u(s)))(\partial_{s}u)ds
\\
 & =4\sup_{t\in\mathbb{S}_{\tau}}\left\Vert J_{t}\right\Vert _{L_{G_{g}}^{\infty}}^{2}\left(\mathcal{A}_{H,\sigma}(x_{-})-\mathcal{A}_{H,\sigma}(x_{+})\right).
\end{align*}

\end{proof}
The proof of Claim 2 is somewhat trickier, and we state this below
as a separate lemma. It is this lemma that explains why in our case
the constant $\delta_{0}(H,\sigma,g)$ enters the picture. 
\begin{lem}
\label{lem:claim 2}Fix $g\in\mathcal{R}_{\sigma}(M)$. Assume $\tau\left|\delta\right|<\delta_{0}(H,\sigma,g)$
and $\alpha\in[S^{1},M]$ is a $\sigma$-atoroidal class. Fix $\mathbf{J}=(J_{t})_{t\in\mathbb{S}_{\tau}}\subseteq\mathcal{U}_{g}$.
Then for all $a\in\mathbb{R}$ there exists a constant $S=S(a)>0$
such that for any $u=(q,p)\in\mathcal{M}_{\tau}^{\alpha}(-\infty,a,H,\delta\sigma,\mathbf{J})$
one has 
\[
\left\Vert p(s,\cdot)\right\Vert _{L_{g}^{2}(\mathbb{S}_{\tau})}\leq S\left(1+\left\Vert \partial_{s}u(s,\cdot)\right\Vert _{L_{G_{g}}^{2}(\mathbb{S}_{\tau})}\right).
\]
\end{lem}
\begin{proof}
We begin with the more difficult case where $\sigma$ satisfies \textbf{($\boldsymbol{\sigma}$1)}
but not \textbf{($\boldsymbol{\sigma}$0)}, so that by \eqref{eq:delta0},
we have 
\[
\delta_{0}(H,\sigma,g)=\frac{\eta_{1}(H,g)}{2C_{0}\eta_{2}(H,g)}.
\]
Set
\[
T:=\sup_{t\in\mathbb{S}_{\tau}}\left\Vert J_{t}\right\Vert _{L_{G_{g}}^{\infty}}.
\]
Fix $u=(q,p)\in\mathcal{M}_{\tau}^{\alpha}(-\infty,a,H,\sigma,\mathbf{J})$
as in the statement of the lemma. Observe that by \eqref{eq:dh z}
we have: 
\begin{align*}
\lambda(\partial_{t}u) & =\lambda(X_{H,\sigma}(t,u))+\lambda(J_{t}(u)\partial_{s}u)\\
 & =dH(t,u)Z(u)+d\lambda(Z(u),J_{t}(\partial_{s}u))\\
 & \geq dH(t,u)Z(u)-T\left|p\right|G_{g}(\partial_{s}u,\partial_{s}u)^{1/2}.
\end{align*}
Thus if $H$ satisfies \textbf{(H1) }with
respect to $h_{1}>0$ and $k_{1}\ge 0$, then 
\[
\lambda(\partial_{t}u)-H(t,u)\geq h_{1}\left|p\right|^{2}-k_{1}-T\left|p\right|G_{g}(\partial_{s}u,\partial_{s}u)^{1/2},
\]
and hence 
\[
\mathcal{A}_{H}(u(s,\cdot))\geq h_{1}\left\Vert p(s,\cdot)\right\Vert _{L_{g}^{2}(\mathbb{S}_{\tau})}^{2}-k_{1}\tau-T\left\Vert p(s,\cdot)\right\Vert _{L_{g}^{2}(\mathbb{S}_{\tau})}\left\Vert \partial_{s}u(s,\cdot)\right\Vert _{L_{G_{g}}^{2}(\mathbb{S}_{\tau})}.
\]
Taking horizontal components of the equation
\[
\partial_{t}u=J_{t}(u)\partial_{s}u+X_{H,\sigma}(t,u)
\]
gives
\[
\partial_{t}q=(J_{t}(u)\partial_{s}u)^{h}+\nabla_{p}H(t,q,p),
\]
and hence if $H$ satisfies the second of the two conditions needed
for \textbf{(H2) }with $h_{2}>0$ and $k_{2} \ge 0$ then 
\begin{align*}
\left|\partial_{t}q\right|^{2} & \leq2\left|(J_{t}(u)\partial_{s}u)^{h}\right|^{2}+2\left|\nabla_{p}H(t,q,p)\right|^{2}\\
 & \leq2\left\Vert J_{t}\right\Vert _{L_{G_{g}}^{\infty}}^{2}G_{g}(\partial_{s}u,\partial_{s}u)+2h_{2}\left|p\right|^{2}+2k_{2}.
\end{align*}
Thus 
\begin{align*}
\left(\int_{0}^{\tau}\left|\partial_{t}q(s,\cdot)\right|dt\right)^{2} & \leq\tau\int_{0}^{\tau}\left|\partial_{t}q(s,\cdot)\right|^{2}dt\\
 & \leq2\tau T^{2}\left\Vert \partial_{s}u(s,\cdot)\right\Vert _{L_{G_{g}}^{2}(\mathbb{S}_{\tau})}^{2}+2\tau h_{2}\left\Vert p(s,\cdot)\right\Vert _{L_{g}^{2}(\mathbb{S}_{\tau})}^{2}+2\tau k_{2}.
\end{align*}
Thus by Lemma \ref{lem:BF-2}, 
\begin{align*}
\left|\mathcal{A}_{\delta\sigma}(q(s,\cdot))\right| & \leq\left|\delta\right|\left(C_{0}\left(\int_{0}^{\tau}\left|\partial_{t}q(s,\cdot)\right|dt\right)^{2}+C_{1}\right)\\
 & \leq2\left|\delta\right|C_{0}\tau T^{2}\left\Vert \partial_{s}u(s,\cdot)\right\Vert _{L_{G_{g}}^{2}(\mathbb{S}_{\tau})}^{2}+2\left|\delta\right|C_{0}\tau h_{2}\left\Vert p(s,\cdot)\right\Vert _{L_{g}^{2}(\mathbb{S}_{\tau})}^{2}+\left|\delta\right|(2C_{0}\tau k_{2}+C_{1}),
\end{align*}
and hence

\begin{align*}
a & \geq\mathcal{A}_{H,\delta\sigma}(u(s,\cdot))\\
 & \geq\mathcal{A}_{H}(u(s,\cdot))-\left|\mathcal{A}_{\delta\sigma}(q(s,\cdot))\right|\\
 & \geq(h_{1}-2\left|\delta\right|C_{0}\tau h_{2})\left\Vert p(s,\cdot)\right\Vert _{L_{g}^{2}(\mathbb{S}_{\tau})}^{2}-T\left\Vert p(s,\cdot)\right\Vert _{L_{g}^{2}(\mathbb{S}_{\tau})}\left\Vert \partial_{s}u(s,\cdot)\right\Vert _{L_{G_{g}}^{2}(\mathbb{S}_{\tau})}\\
 & -2\left|\delta\right|C_{0}\tau T^{2}\left\Vert \partial_{s}u(s,\cdot)\right\Vert _{L_{G_{g}}^{2}(\mathbb{S}_{\tau})}^{2}-k_{1}\tau-\left|\delta\right|(2C_{0}\tau k_{2}+C_{1}).
\end{align*}
Using the fact that for any $c,d,\mu>0$ it holds that 
\[
cd\leq\mu c^{2}+\frac{1}{4\mu}d^{2},
\]
we have that for any $\mu>0$ it holds that
\begin{align*}
a & \geq(h_{1}-2\left|\delta\right|\tau C_{0} h_{2}-\mu)\left\Vert p(s,\cdot)\right\Vert _{L_{g}^{2}(\mathbb{S}_{\tau})}^{2}-\left(2\left|\delta\right|C_{0}\tau T^{2}+\frac{1}{4\mu}T^{2}\right)\left\Vert \partial_{s}u(s,\cdot)\right\Vert _{L_{G_{g}}^{2}(\mathbb{S}_{\tau})}^{2}\\
 & -k_{1}\tau-\left|\delta\right|(2C_{0}\tau k_{2}+C_{1}).
\end{align*}
Our choice of $\delta$ implies that 
\[
h_{1}-2\left|\delta\right|\tau C_{0} h_{2}>0,
\]
and hence for suitably small $\mu$ we obtain an equality of the desired
form.

Finally consider the case where $\sigma$ satisfies the stronger condition\textbf{
($\boldsymbol{\sigma}$0)}.\textbf{ }In this case $\widetilde{\sigma}$
admits a bounded primitive $\theta$, and Lemma \ref{lem:BF-2} can
be upgraded to a \textbf{linear isoperimetric inequality} - see \cite[Lemma 4.4]{BaeFrauenfelder2010}.
It is then easy to improve the proof above to work for any $\delta\in\mathbb{R}$,
and we omit the details.
\end{proof}
We have now verified Claim 2 of Lemma 1.12 in \cite{AbbondandoloSchwarz2006}.
As discussed above, the remaining parts of the proof of Lemma 1.12
go through without change in our situation, and thus this concludes
the proof of equation \eqref{eq:first stage}, and hence also of Theorem
\ref{thm:AS1}.
\begin{rem}
\label{finite critical pts}This argument also proves that if $\tau\left|\delta\right|<\delta_{0}(H,\sigma,g)$
and the triple $(H,\delta\sigma,\alpha)$ satisfies Condition \textbf{(N)},\textbf{
}then for any $a\in\mathbb{R}$, there are at most finitely many critical
points $x\in\mathcal{P}_{\tau}^{\alpha}(H,\delta\sigma)$ with $\mathcal{A}_{H,\delta\sigma}(x)\leq a$.
Indeed, the proof shows that if 
\[
\mathbb{P}:=\left\{ x\in\mathcal{P}_{\tau}^{\alpha}(H,\delta\sigma)\,:\,\mathcal{A}_{H,\delta\sigma}(x)\leq a\right\} ,
\]
then there exists a uniform bound on $\left\Vert p\right\Vert _{L_{g}^{2}(\mathbb{S}_{\tau})}^{2}$
for all $x=(q,p)\in\mathbb{P}$. 

Since 
\[
\left|\dot{x}\right|=\left|X_{H,\sigma}(t,x)\right|\leq h_{\sigma,g}\left(1+\left|p\right|^{2}\right)
\]
by \eqref{eq:the condition on X}, we see that $\mathbb{P}$ is bounded
in $W_{g}^{1,1}$, and hence in $L_{g}^{\infty}$. In particular,
the set 
\[
\{x(0)\,:\, x\in\mathbb{P}\}
\]
is precompact in $T^{*}M$, and since it is discrete by Condition
\textbf{(N)}, it is finite.
\end{rem}

\subsection{\label{sub:homotopies}Invariance}

The following result completes the proof of Theorem A from the Introduction,
whose proof is similar to \cite[Lemma 1.21]{AbbondandoloSchwarz2006}
and \cite[Theorem 2.7]{BaeFrauenfelder2010}. Indeed, to obtain Theorem A from Theorem \ref{thm:third statement}, 
simply take $\sigma_0 = \sigma$ and $\sigma_1 = 0$, and apply Theorem \ref{thm:[Abbondandolo-Schwarz-,-Salamon-}.
\begin{thm}
\label{thm:third statement}\textbf{\emph{(Invariance of Floer homology under homotopies)}}
Fix a Riemannian metric $g$ on $M$,
$\alpha\in[S^{1},M]$ and $\tau>0$. Suppose we are given:
\begin{enumerate}
\item 2-forms $\sigma_{0}$ and $\sigma_{1}$ that both satisfy \textbf{\emph{(}}\emph{$\boldsymbol{\sigma}$}\textbf{\emph{1)}}
and are such that $\alpha$ is both $\sigma_{0}$-atoroidal and $\sigma_{1}$-atoroidal,
and such that $g\in\mathcal{R}_{\sigma_{0}}(M)\cap\mathcal{R}_{\sigma_{1}}(M)$.
Set 
\[
\sigma_{s}:=(1-s)\sigma_{0}+s\sigma_{1}.
\]

\item Hamiltonians $H_{0}$ and $H_{1}$ satisfying the Abbondandolo-Schwarz
growth conditions. Set
\[
H_{s}:=(1-s)H_{0}+sH_{1},
\]

\end{enumerate}
Choose a smooth function $\delta:[0,1]\rightarrow\mathbb{R}$ such
that
\[
\tau\left|\delta(s)\right|<\delta_{0}(H_{s},\sigma_{s},g)
\]
 for each $s\in[0,1]$, and suppose that both 
$(H_{0},\delta(0)\sigma_{0},\alpha)$ and $(H_{1},\delta(1)\sigma_{1},\alpha)$ satisfy Condition\emph{
}\textbf{\emph{(N)}}. Then there exists a continuation map 
\[
\Psi:CF_{*}^{\alpha}(H_{0},\delta(0)\sigma_{0},\tau)\rightarrow CF_{*}^{\alpha}(H_{1},\delta(1)\sigma_{1},\tau)
\]
inducing an isomorphism 
\[
\psi:HF_{*}^{\alpha}(H_{0},\delta(0)\sigma_{0},\tau)\rightarrow HF_{*}^{\alpha}(H_{1},\delta(1)\sigma_{1},\tau).
\]

\end{thm}
Before getting started on the proof, we will introduce some notation.
Our assumption
\begin{equation}
\tau\left|\delta(s)\right|<\delta_{0}(H_{s},\sigma_{s},g)\ \ \ \mbox{for all }s\in[0,1]\label{eq:key smallness assumption-1-1}
\end{equation}
 implies that we can choose bounded functions $\eta_{1}(s),\eta_{2}(s)$
and constants $k_{1},k_{2}\geq0$ and $\chi>0$ such that for all
$s\in[0,1]$:
\begin{enumerate}
\item $H_{s}$ satisfies \textbf{(H1) }with respect to $\eta_{1}(s)$ and
$k_{1}$;
\item $H_{s}$ satisfies \eqref{eq:p part of H2} with respect to $\eta_{2}(s)$
and $k_{2}$;
\item If $C_{0}(\sigma_{s},g)$ and $C_{1}(\sigma_{s},g,\alpha)$ denote
the constants associated to $\sigma_{s}$ from Lemma \ref{lem:BF-2}
then 
\begin{equation}
\eta_{1}(s)-2\left|\delta(s)\right|\tau C_{0}(\sigma_{s},g)\eta_{2}(s)>\chi\ \ \ \mbox{for all }s\in[0,1].\label{eq:the key to positivity}
\end{equation}

\end{enumerate}
Set
\[
\eta_{1}:=\max_{s\in[0,1]}\eta_{1}(s),\ \ \ \eta_{2}:=\max_{s\in[0,1]}\eta_{2}(s);
\]
\[
C_{0}:=\max_{s\in[0,1]}C_{0}(\sigma_{s},g),\ \ \ C_{1}:=\max_{s\in[0,1]}C_{1}(\sigma_{s},g,\alpha);
\]
\[
d:=\max_{s\in[0,1]}\left|\delta(s)\right|.
\]
Now fix $\varepsilon>0$, which we will specify precisely later. Choose
a natural number 
\begin{equation}
N\geq\frac{2\max_{s\in[0,1]}\left|\delta'(s)\right|}{\varepsilon},\label{eq:def of capital N}
\end{equation}
and choose a subdivision $0=r_{0}<r_{1}<\dots<r_{N}=1$ such that
$\left|r_{i}-r_{i+1}\right|<2/N$ for each $i=0,\dots,N-1$ and such
that for each $i=0,\dots,N-1$ the following two inequalities hold \footnote{That it is possible to choose such a subdivision so that the first
inequality holds is explained in \cite[p289]{AbbondandoloSchwarz2006},
and uses the fact that both $H_{0}$ and $H_{1}$ satisfy \textbf{(H2)},
and that $M$ is compact.
}:
\begin{equation}
\begin{cases}
\left|H_{r_{i+1}}(t,q,p)-H_{r_{i}}(t,q,p)\right|\leq\varepsilon\left(1+\left|p\right|^{2}\right);\\
C_{0}(\sigma_{r_{i+1}}-\sigma_{r_{i}},g)<\varepsilon.
\end{cases}\label{eq:choosing N}
\end{equation}
Let $\beta:\mathbb{R}\rightarrow[0,1]$ denote a smooth cut-off function
such that $\beta(s)\equiv0$ for $s\leq0$ and $\beta(s)\equiv1$
for $s\geq1$, with $0\leq\beta'(s)\leq2$ for all $s\in\mathbb{R}$.
Now define:
\[
F_{s}^{i}:=H_{r_{i}}+\beta(s)(H_{r_{i+1}}-H_{r_{i}});
\]
\[
\nu_{s}^{i}:=\sigma_{r_{i}}+\beta(s)(\sigma_{r_{i+1}}-\sigma_{r_{i}});
\]
\[
f_{i}(s):=\delta(r_{i}+\beta(s)(r_{i+1}-r_{i}));
\]
\[
\omega_{s}^{i}:=d\lambda+f{}_{i}(s)\pi^{*}\nu_{s}^{i}.
\]
Note that by \eqref{eq:def of capital N}, 
\[
\max_{s\in[0,1]}\left|f_{i}'(s)\right|<2\varepsilon\ \ \ \mbox{for all }i\in\{0,1,\dots,N-1\}.
\]
Let
\[
\mathcal{A}^{i}:\Lambda_{\tau}^{\alpha}T^{*}M\rightarrow\mathbb{R}
\]
be defined by 
\[
\mathcal{A}^{i}(x):=\mathcal{A}_{H_{r_{i}},\delta(r_{i})\sigma_{r_{i}}}(x)=\mathcal{A}_{H_{r_{i}}}(x)+\mathcal{A}_{\delta(r_{i})\sigma_{r_{i}}}(\pi\circ x),
\]
and let
\[
\mathcal{A}_{s}^{i}:\Lambda_{\tau}^{\alpha}T^{*}M\rightarrow\mathbb{R}
\]
be defined by 
\[
\mathcal{A}_{s}^{i}(x):=\mathcal{A}_{F_{s}^{i},f_{i}(s)\nu_{s}^{i}}(x)=\mathcal{A}_{F_{s}^{i}}(x)+\mathcal{A}_{f_{i}(s)\nu_{s}^{i}}(\pi\circ x).
\]
Fix $\mathbf{J}=(J_{t})_{t\in\mathbb{S}_{\tau}}\subseteq\mathcal{V}_{g}$
(where $\mathcal{V}_{g}$ is as in the statement of Theorem \ref{thm:AS1}).\newline

Given $i\in\{0,1,\dots,N-1\}$ and $-\infty<a\leq b<\infty$, denote
by 
\[
\mathcal{N}_{\tau}^{\alpha}(a,b,F_{s}^{i},f_{i}(s)\nu_{s}^{i},\mathbf{J})
\]
the set of maps $u\in C^{\infty}(\mathbb{R}\times\mathbb{S}_{\tau},T^{*}M)$
that satisfy the $s$-dependent Floer equation
\[
\partial_{s}u+\nabla_{\mathbf{J}}\mathcal{A}_{s}^{i}(u)=0
\]
and which satisfy
\[
a\leq\mathcal{A}_{s}^{i}(u(s,t))\leq b\ \ \ \mbox{for all }(s,t)\in\mathbb{R}\times\mathbb{S}_{\tau}.
\]

The following statement constitutes most of the work needed to prove
Theorem \ref{thm:third statement}.
\begin{lem}
\label{lem:compact set-1}If $\varepsilon>0$ is sufficiently small
then given any $i\in\{0,1,\dots,N-1\}$ and any $-\infty<a\leq b<\infty$
there exists a compact set $K_{i}=K_{i}(a,b,\mathbf{J})\subseteq T^{*}M$
such that for all $u\in\mathcal{N}_{\tau}^{\alpha}(a,b,F_{s}^{i},f_{i}(s)\nu_{s}^{i},\mathbf{J})$
one has $u(\mathbb{R}\times\mathbb{S}_{\tau})\subseteq K_{i}.$\end{lem}
\begin{proof}
Fix $i\in\{0,1,\dots,N-1\}$, and fix $u=(q,p)\in\mathcal{N}_{\tau}^{\alpha}(a,b,F_{s}^{i},f_{i}(s)\nu_{s}^{i},\mathbf{J})$.
Firstly, note that by Lemma \ref{lem:BF-2} we have that for all $s\in\mathbb{R}$,
\begin{equation}
\left|\mathcal{A}_{\nu_{s}^{i}}(q(s,\cdot))\right|\leq C_{0}\left(\int_{0}^{\tau}\left|\partial_{t}q(s,\cdot)\right|dt\right)^{2}+C_{1};\label{eq:1}
\end{equation}
\begin{equation}
\left|\mathcal{A}_{(\sigma_{r_{i+1}}-\sigma_{r_{i}})}(q(s,\cdot))\right|\leq\varepsilon\left(\int_{0}^{\tau}\left|\partial_{t}q(s,\cdot)\right|dt\right)^{2}+C_{2},\label{eq:2}
\end{equation}
for some constant $C_{2}>0$, where the second equation used \eqref{eq:choosing N}. 

The key term we wish to estimate is:
\[
\Delta(u):=\int_{-\infty}^{\infty}\left|\left(\frac{\partial}{\partial s}\mathcal{A}_{s}^{i}\right)(u(s,\cdot))\right|ds.
\]
We compute
\begin{align*}
\left|\left(\frac{\partial}{\partial s}\mathcal{A}_{s}^{i}\right)(u(s,\cdot))\right| & =\left|-\int_{0}^{1}\left(\frac{\partial}{\partial s}F_{s}^{i}\right)(u(s,t))dt+\frac{\partial}{\partial s}\mathcal{A}_{f_{i}(s)\nu_{s}^{i}}(q(s,\cdot))\right|\\
 & \leq\beta'(s)\int_{0}^{1}\left|(H_{r_{i+1}}-H_{r_{i}})(t,u)\right|dt+\left|\frac{\partial}{\partial s}\mathcal{A}_{f_{i}(s)\nu_{s}^{i}}(q(s,\cdot))\right|.
\end{align*}
We can estimate the first term from \eqref{eq:choosing N} by
\[
\beta'(s)\int_{0}^{1}\left|(H_{r_{i+1}}-H_{r_{i}})(t,u)\right|dt\leq2\varepsilon\left(1+\left\Vert p(s,\cdot)\right\Vert _{L_{g}^{2}(\mathbb{S}_{\tau})}^{2}\right)
\]
 As for the second term, we compute using \eqref{eq:1} and \eqref{eq:2}
that
\begin{align*}
\left|\frac{\partial}{\partial s}\mathcal{A}_{f_{i}(s)\nu_{s}^{i}}(u(s,\cdot))\right| & =\left|f'_{i}(s)\mathcal{A}_{\nu_{s}^{i}}(q(s,\cdot))+\beta'(s)f_{i}(s)\mathcal{A}_{(\sigma_{r_{i+1}}-\sigma_{r_{i}})}(q(s,\cdot))\right|\\
 & \leq2\varepsilon\left(C_{0}\left(\int_{0}^{\tau}\left|\partial_{t}q(s,\cdot)\right|dt\right)^{2}+C_{1}\right)+2d\left(\varepsilon\left(\int_{0}^{\tau}\left|\partial_{t}q(s,\cdot)\right|dt\right)^{2}+C_{2}\right)\\
 & \leq2\varepsilon(C_{0}+d)\left(\int_{0}^{\tau}\left|\partial_{t}q(s,\cdot)\right|dt\right)^{2}+2\varepsilon C_{1}+2dC_{2}.
\end{align*}
Arguing as in the proof of Lemma \ref{lem:claim 2}, we have
\[
\left(\int_{0}^{\tau}\left|\partial_{t}q(s,\cdot)\right|dt\right)^{2}\leq2\tau T^{2}\left\Vert \partial_{s}u(s,\cdot)\right\Vert _{L_{G_{g}}^{2}(\mathbb{S}_{\tau})}^{2}+2\tau\eta_{2}\left\Vert p(s,\cdot)\right\Vert _{L_{g}^{2}(\mathbb{S}_{\tau})}^{2}+2\tau k_{2},
\]
where as before, 
\[
T:=\sup_{t\in\mathbb{S}_{\tau}}\left\Vert J_{t}\right\Vert _{L_{G_{g}}^{\infty}}.
\]
Thus
\begin{align*}
\left|\frac{\partial}{\partial s}\mathcal{A}_{f_{i}(s)\nu_{s}^{i}}(q	(s,\cdot))\right| & \leq4\varepsilon(C_{0}+d)\tau T^{2}\left\Vert \partial_{s}u(s,\cdot)\right\Vert _{L_{G_{g}}^{2}(\mathbb{S}_{\tau})}^{2}+4\varepsilon(C_{0}+d)\tau\eta_{2}\left\Vert p(s,\cdot)\right\Vert _{L_{g}^{2}(\mathbb{S}_{\tau})}^{2}.\\
 & +4\varepsilon(C_{0}+d)\tau k_{2}+2\varepsilon C_{1}+2dC_{2}.
\end{align*}
 Putting this together and integrating we conclude 
\begin{align*}
\Delta(u) & \leq\underset{:=c_{1}}{\underbrace{(2\varepsilon+4\varepsilon(C_{0}+d)\tau\eta_{2})}}\left\Vert p\right\Vert _{L_{g}^{2}([0,1]\times\mathbb{S}_{\tau})}^{2}+\underset{:=c_{2}}{\underbrace{4\varepsilon(C_{0}+d)\tau T^{2}}}\left\Vert \partial_{s}u\right\Vert _{L_{G_{g}}^{2}([0,1]\times\mathbb{S}_{\tau})}^{2}\\
 & +\underset{:=c_{3}}{\underbrace{2\varepsilon+4\varepsilon(C_{0}+d)\tau k_{2}+2\varepsilon C_{1}+2dC_{2}}}.
\end{align*}
Arguing as in Lemma \ref{lem:claim 1} we have 
\begin{align*}
\left\Vert \partial_{s}u\right\Vert _{L_{G_{g}}^{2}(\mathbb{R}\times\mathbb{S}_{\tau})}^{2} & \leq4T^{2}(b-a+\Delta(u))\\
 & \leq4T^{2}c_{1}\left\Vert p\right\Vert _{L_{g}^{2}([0,1]\times\mathbb{S}_{\tau})}^{2}+4T^{2}c_{2}\left\Vert \partial_{s}u\right\Vert _{L_{G_{g}}^{2}([0,1]\times\mathbb{S}_{\tau})}^{2}+4T^{2}(b-a+c_{3}),
\end{align*}
and thus provided $\varepsilon>0$ is small enough such that 
\[
4T^{2}c_{2}\leq\frac{1}{2},
\]
 we conclude that 
\begin{equation}
\left\Vert \partial_{s}u\right\Vert _{L_{G_{g}}^{2}(\mathbb{R}\times\mathbb{S}_{\tau})}^{2}\leq8T^{2}c_{1}\left\Vert p\right\Vert _{L_{g}^{2}([0,1]\times\mathbb{S}_{\tau})}^{2}+8T^{2}(b-a+c_{3}).\label{eq:energy-1}
\end{equation}
Similarly one has 
\begin{align}
\sup_{s\in\mathbb{R}}\mathcal{A}_{s}^{i}(u(s,\cdot)) & \leq b+\Delta(u)\nonumber \\
 & \leq b+c_{1}\left\Vert p\right\Vert _{L_{g}^{2}([0,1]\times\mathbb{S}_{\tau})}^{2}+c_{2}\left\Vert \partial_{s}u\right\Vert _{L_{G_{g}}^{2}([0,1]\times\mathbb{S}_{\tau})}^{2}+c_{3}.\label{eq:action-1}
\end{align}
Arguing as in the proof of Lemma \ref{lem:claim 2} we discover that
\begin{multline*}
c_{1}\left\Vert p\right\Vert _{L_{g}^{2}([0,1]\times\mathbb{S}_{\tau})}^{2}+c_{2}\left\Vert \partial_{s}u\right\Vert _{L_{G_{g}}^{2}([0,1]\times\mathbb{S}_{\tau})}^{2}+c_{3}+b\geq\mathcal{A}_{s}^{i}(u(s,\cdot))\\
=\int_{0}^{\tau}( \lambda(X_{F_{s}^{i},f_{i}(s)\nu_{s}^{i}}(t,u))-F_{s}(t,u))dt+\mathcal{A}_{f_{i}(s)\nu_{s}^{i}}(q(s,\cdot))\\
\overset{(*)}{\geq}(\chi-\mu)\left\Vert p(s,\cdot)\right\Vert _{L_{g}^{2}(\mathbb{S}_{\tau})}^{2}-\left(2\left|\delta\right|C_{0}\tau T^{2}+\frac{1}{4\mu}T^{2}\right)\left\Vert \partial_{s}u(s,\cdot)\right\Vert _{L_{G_{g}}^{2}(\mathbb{S}_{\tau})}^{2}-k_{1}\tau-d(2C_{0}\tau k_{2}+C_{1}),
\end{multline*}
where $\mu>0$ is any positive number and $(*)$ used \eqref{eq:the key to positivity}.
Take $\mu=\chi/2$. Integrating this expression over $[0,1]$ and
rearranging gives 
\begin{align*}
\left(\frac{\chi}{2}-c_{1}\right)\left\Vert p\right\Vert _{L_{g}^{2}([0,1]\times\mathbb{S}_{\tau})}^{2} & \leq\left(c_{2}+\left(2\left|\delta\right|C_{0}\tau T^{2}+\frac{1}{2\chi}T^{2}\right)\right)\left\Vert \partial_{s}u\right\Vert _{L_{G_{g}}^{2}(\mathbb{R}\times\mathbb{S}_{\tau})}^{2}\\
 & +b+c_{3}+k_{1}\tau+d(2C_{0}\tau k_{2}+C_{1}).
\end{align*}
Substituting in the expression \eqref{eq:energy-1} for $\left\Vert \partial_{s}u\right\Vert _{L_{G_{g}}^{2}(\mathbb{R}\times\mathbb{S}_{\tau})}^{2}$
we obtain 
\begin{multline*}
\underset{:=c_{4}}{\underbrace{\left(\frac{\chi}{2}-c_{1}-8T^{2}c_{1}\left(c_{2}+\left(2\left|\delta\right|C_{0}\tau T^{2}+\frac{1}{2\chi}T^{2}\right)\right)\right)}}\left\Vert p\right\Vert _{L_{g}^{2}([0,1]\times\mathbb{S}_{\tau})}^{2}\\
\leq\underset{:=c_{5}}{\underbrace{8T^{2}(b-a+c_{3})\left(c_{2}+\left(2\left|\delta\right|C_{0}\tau T^{2}+\frac{1}{2\chi}T^{2}\right)\right)+b+c_{3}+k_{1}\tau+d(2C_{0}\tau k_{2}+C_{1})}.}
\end{multline*}
We can choose $\varepsilon>0$ sufficiently small such%
\footnote{Here of course it is important to note that this choice can be made
\textbf{independently }of both $N$ and $i$.%
} that $c_{4}>\chi/4$. Assuming this is so, we have proved that for
any $u=(q,p)\in\mathcal{N}_{\tau}^{\alpha}(a,b,F_{s}^{i},f_{i}(s)\nu_{s}^{i},\mathbf{J})$
one has
\[
\left\Vert p\right\Vert _{L_{g}^{2}([0,1]\times\mathbb{S}_{\tau})}^{2}\leq\frac{4c_{5}}{\chi}.
\]
Feeding this into \eqref{eq:energy-1} and \eqref{eq:action-1} we
find constants $c_{6},c_{7}>0$ such that for all such maps $u$,
\[
\left\Vert \partial_{s}u\right\Vert _{L_{G_{g}}^{2}(\mathbb{R}\times\mathbb{S}_{\tau})}\leq c_{6},\ \ \ \sup_{s\in\mathbb{R}}\mathcal{A}_{s}^{i}(u(s,\cdot))\leq c_{7}.
\]
This proves the analogue of Lemma \ref{lem:claim 1}, and allows us
to prove the analogue of Lemma \ref{lem:claim 2}, for elements of
$\mathcal{N}_{\tau}^{\alpha}(a,b,F_{s}^{i},f_{i}(s)\nu_{s}^{i},\mathbf{J})$.
We can proceed exactly as in the proof of Theorem \ref{thm:AS1} to
obtain the desired compact set $K_{i}$. This completes the proof
of Lemma \ref{lem:compact set-1}.
\end{proof}
Armed with Lemma \ref{lem:compact set-1}, the proof of Theorem \ref{thm:third statement}
is very standard. 
\begin{proof}
\emph{(of Theorem \ref{thm:third statement})}

Fix $N\in\mathbb{N}$ such that there exists a subdivision $0=r_{0}<r_{1}<\dots<r_{N}=1$
with the property that \eqref{eq:choosing N} holds for some $\varepsilon>0$
small enough such that Lemma \ref{lem:compact set-1} holds for each
$i=0,1,\dots,N-1$. After possibly making additional arbitrarily small perturbations
of $H_{s}$ for $s$ near $r_{i}$, for each $i=1,2,\dots,N-1$ (which
for simplicity we omit from our notation), we may assume that $(H_{r_{i}},\delta(r_{i})\sigma_{r_{i}},\alpha)$
satisfies Condition \textbf{(N) }for each $i=0,1,\dots,N$. 

Under these assumptions we define for each $i=0,1,\dots,N-1$ a continuation
map 
\[
\Psi_{i}(\mathbf{J}):CF_{*}^{\alpha}(H_{r_{i}},\delta(r_{i})\sigma_{r_{i}},\tau)\rightarrow CF_{*}^{\alpha}(H_{r_{i+1}},\delta(r_{i+1})\sigma_{r_{i+1}},\tau)
\]
by
\[
\Psi_{i}(\mathbf{J})(x):=\sum_{y\in\mathcal{P}_{\tau}^{\alpha}(H_{r_{i+1}},\delta(r_{i+1})\sigma_{r_{i+1}})_{k}}n_{i}(x,y)y,\ \ \ x\in\mathcal{P}_{\tau}^{\alpha}(H_{r_{i}},\delta(r_{i})\sigma_{r_{i}})_{k},
\]
where
\[
n_{i}(x,y):=\#_{2}\mathcal{N}_{\tau}^{\alpha}(x,y,F_{s}^{i},f_{i}(s)\nu_{s}^{i},\mathbf{J}),
\]
and $\mathcal{N}_{\tau}^{\alpha}(x,y,F_{s}^{i},f_{i}(s)\nu_{s}^{i},\mathbf{J})$
denotes the (finite) set of maps $u:\mathbb{R}\times\mathbb{S}_{\tau}\rightarrow T^{*}M$
satisfying 
\[
\partial_{s}u+\nabla_{\mathbf{J}}\mathcal{A}_{s}^{i}(u)=0,
\]
and which submit to the asymptotic conditions
\[
\lim_{s\rightarrow\infty}u(s,t)=x(t),\ \ \ \lim_{s\rightarrow-\infty}u(s,t)=y(t),\ \ \ \lim_{s\rightarrow\pm\infty}\partial_{s}u(s,t)=0.
\]
Standard Floer-theoretical arguments (see for instance \cite{Salamon1999})
tell us that the $\Psi_{i}(\mathbf{J})$ are chain maps that induce
isomorphisms 
\[
\psi_{i}:HF_{*}^{\alpha}(H_{r_{i}},\delta(r_{i})\sigma_{r_{i}},\tau)\rightarrow HF_{*}^{\alpha}(H_{r_{i+1}},\delta(r_{i+1})\sigma_{r_{i+1}},\tau)
\]
for $i=0,1,\dots,N-1$ on homology. The chain map $\Psi$ from the
statement of the theorem is then defined as the composition 
\[
\Psi:=\Psi_{N-1}(\mathbf{J})\circ\dots\circ\Psi_{1}(\mathbf{J})\circ\Psi_{0}(\mathbf{J}).
\]

\end{proof}
\appendix

\section{The Lagrangian Framework}

In this Appendix we outline an alternative approach to obtaining some
of the results of this paper \textbf{without} using the machinery of Floer
homology. Roughly speaking, this method can be used to recover all
of the results proved in this paper for a more restricted class of
Hamiltonian systems: the so-called \textbf{convex quadratic growth
Hamiltonians}, which are those Hamiltonians $H:\mathbb{S}_{\tau}\times T^{*}M\rightarrow\mathbb{R}$
which satisfy the Abbondandolo-Schwarz growth conditions and \textbf{in
addition }are \textbf{strictly fibrewise convex}. 

Given such a Hamiltonian $H$, the idea is to study the \textbf{Lagrangian
action functional }$\mathcal{S}_{L,\delta\sigma}$ on the atoroidal
components of the (completed) $\tau$-periodic loop space of $M$,
where $L$ is the \textbf{Fenchel dual Lagrangian }of $H$. The key
point is to show that (for $\tau\left|\delta\right|$ sufficiently
small), the functional $\mathcal{S}_{L,\delta\sigma}$ satisfies the
\textbf{Palais-Smale condition}, and this allows one to construct
the \textbf{Morse complex }of $\mathcal{S}_{L,\delta\sigma}$.

\subsection{\label{sub:The-Lagrangian-action}The Lagrangian action functional}

Fix a Riemannian metric $g$ on $M$. Suppose $L\in C^{\infty}(TM,\mathbb{R})$.
Then $dL(q,v)\in T_{(q,v)}^{*}TM$, and thus its gradient $\nabla L(q,v)$
(with respect to the $G_{g}$-metric on $TM$) lies in $T_{(q,v)}TM$.
Thus we can speak of the horizontal and vertical components 
\[
\nabla_{q}L(q,v):=\nabla L(q,v)^{h}\in T_{q}M;
\]
\[
\nabla_{v}L(q,v):=\nabla L(q,v)^{v}\in T_{q}M.
\]
Thinking of $\nabla_{q}L$ as a map $TM\rightarrow TM$ (so its derivative
is a map $d(\nabla_{q}L):TTM\rightarrow TTM$), we define
\[
\nabla_{qq}L(q,v)(w):=d(\nabla_{q}L)(q,v)(\xi_{w})^{v},
\]
where $\xi_{w}\in T_{(q,v)}TM$ is the unique vector such that $\xi_{w}^{h}=w$
and $\xi_{w}^{v}=0$. Similarly we define 
\[
\nabla_{qv}L(q,v)(w):=d(\nabla_{q}L)(q,v)(\zeta_{w})^{v},
\]
where this time $\zeta_{w}\in T_{(q,v)}TM$ is the unique vector such
that $\zeta_{w}^{h}=0$ and $\zeta_{w}^{v}=w$. We define maps $\nabla_{qv}L$
and $\nabla_{vv}L$ in exactly the same way, starting with $\nabla_{v}L$
instead of $\nabla_{q}L$. Note that the operator $\nabla_{vv}L(q,v):T_{q}M\rightarrow T_{q}M$
coincides with the second derivative of the map $v\mapsto L(q,v)$
in the vector space $T_{q}M$. If $L$ is time-dependent then these
notations still make sense, with $\nabla_{qq}L(t,q,v):=\nabla_{qq}L_{t}(q,v)$
etc., where $L_{t}(q,v):=L(t,q,v)$.\newline

We will be interested in time-dependent Lagrangians $L\in C^{\infty}(\mathbb{S}_{\tau}\times TM,\mathbb{R})$
that satisfy the following \textbf{convex quadratic growth conditions}:
\begin{lyxlist}{00.00.0000}
\item [{\textbf{(L1)}}] There exists $\ell_{1}>0$ such that for all $(t,q,v)\in\mathbb{S}_{\tau}\times TM$
it holds that 
\[
\nabla_{vv}L(t,q,v)\geq\ell_{1}\mathbb{1}.
\]

\item [{\textbf{(L2)}}] There exists $\ell_{2}>0$ such that for all $(t,q,v)\in\mathbb{S}_{\tau}\times TM$
it holds that 
\[
\left|\nabla_{vv}L(t,q,v)\right|\leq\ell_{2},\ \left|\nabla_{vq}L(t,q,v)\right|\leq\ell_{2}(1+\left|v\right|),\ \left|\nabla_{qq}L(t,q,v)\right|\leq\ell_{2}(1+\left|v\right|^{2}).
\]

\end{lyxlist}
Whilst the constants $\ell_{1}$ and $\ell_{2}$ depend on the choice
of metric $g$ on $M$, the existence of such constants does not (see
\cite[Proposition 3.3.1]{Mazzucchelli2011}). Note that the assumption
\textbf{(L1) }implies that $\nabla_{v}L(t,q,\cdot):T_{q}M\rightarrow T_{q}^{*}M$
is a diffeomorphism for each $(t,q)\in\mathbb{S}_{\tau}\times M$,
and hence we may define the \textbf{Fenchel dual Hamiltonian} $H\in C^{\infty}(\mathbb{S}_{\tau}\times T^{*}M,\mathbb{R})$
by 
\begin{equation}
H(t,q,p):=p(v)-L(t,q,v),\ \ \ \mbox{where }\nabla_{v}L(t,q,v)=p.\label{eq:legendre}
\end{equation}
It is not hard to check that asking $L$ to satisfy \textbf{(L1) }and
\textbf{(L2) }implies that $H$ satisfies the Abbondandolo-Schwarz
growth conditions \textbf{(H1) }and \textbf{(H2)}. Going the other
way round, if $H\in C^{\infty}(\mathbb{S}_{\tau}\times T^{*}M,\mathbb{R})$
satisfies \textbf{(H1) }and \textbf{(H2) }and\textbf{ }in addition
is \textbf{strictly fibrewise convex}, then there is a unique Lagrangian
$L\in C^{\infty}(\mathbb{S}_{\tau}\times TM,\mathbb{R})$ called the
\textbf{Fenchel dual Lagrangian }of $H$ for which $\nabla_{v}L(t,q,\cdot)$
is a diffeomorphism for each $(t,q)\in\mathbb{S}_{\tau}\times M$,
and which is related to $H$ by \eqref{eq:legendre}. Moreover, this
Lagrangian $L$ satisfies \textbf{(L1) }and \textbf{(L2)}.\newline 

Denote by $\mathcal{L}_{\tau}M:=W^{1,2}(\mathbb{S}_{\tau},M)$ the
\textbf{Sobolev completion }of the free loop space $\Lambda_{\tau}M=C^{\infty}(\mathbb{S}_{\tau},M)$,
and as before given $\alpha\in[S^{1},M]$ denote by $\mathcal{L}_{\tau}^{\alpha}M$
the component of $\mathcal{L}_{\tau}M$ belonging to $\alpha$. Unlike
$\Lambda_{\tau}M$, the space $\mathcal{L}_{\tau}M$ carries the structure
of a Hilbert manifold, and therefore is much better suited for doing
Morse homology on. As before we denote by $\left\Vert \cdot\right\Vert _{W_{g}^{1,2}(\mathbb{S}_{\tau})}$
the $W_{g}^{1,2}$-metric on $\mathcal{L}_{\tau}^{\alpha}M$.

In this Appendix we study the \textbf{Lagrangian action functional
}$\mathcal{S}_{L,\sigma}:\mathcal{L}_{\tau}^{\alpha}M\rightarrow\mathbb{R}$
associated to a Lagrangian $L$ satisfying \textbf{(L1) }and \textbf{(L2)},
together with a 2-form $ $$\sigma$ satisfying \textbf{(}$\boldsymbol{\sigma}$\textbf{1)}
on a $\sigma$-atoroidal class $\alpha\in[S^{1},M]$. As with the
Hamiltonian action functional $\mathcal{A}_{H,\sigma}$, the Lagrangian
action functional $\mathcal{S}_{L,\sigma}$ is defined as the sum
\[
\mathcal{S}_{L,\sigma}(q):=\mathcal{S}_{L}(q)+\mathcal{A}_{\sigma}(q),
\]
where $\mathit{\mathcal{S}_{L}}$ is the \textbf{standard }Lagrangian
action functional

\[
\mathcal{S}_{L}(q):=\int_{0}^{\tau}L(t,q(t),\dot{q}(t))dt
\]
(note $\mathcal{S}_{L}$ is defined on all of $\mathcal{L}_{\tau}M$),
and $\mathcal{A}_{\sigma}$ is defined as before (only now on the
completed loop space $\mathcal{L}_{\tau}^{\alpha}M$). 

A standard computation (which does not use assumptions \textbf{(L1)
}and \textbf{(L2)} and only requires that $\alpha$ is a $\sigma$-atoroidal
class) tells us that if $q\in\mathcal{L}_{\tau}^{\alpha}M$ and $(q_{s})_{s\in(-\varepsilon,\varepsilon)}\subseteq\mathcal{L}_{\tau}^{\alpha}M$
is a variation of $q$ with $\frac{\partial}{\partial s}\bigl|_{s=0}q_{s}(t)=:\xi(t)$
then, writing $Y=Y_{\sigma,g}$ for the Lorentz force defined in Section \ref{sub:Preliminaries}, we have:
\begin{equation}
\frac{\partial}{\partial s}\Bigl|_{s=0}\mathcal{S}_{L,\sigma}(q_{s})=\int_{0}^{\tau}\left\langle \nabla_{q}L(t,q,\dot{q}),\xi\right\rangle +\left\langle \nabla_{v}L(t,q,\dot{q}),\nabla_{t}\xi\right\rangle +\left\langle Y(q)\dot{q},\xi\right\rangle dt,\label{eq:differential}
\end{equation}
which we can rewrite as 
\[
\frac{\partial}{\partial s}\Bigl|_{s=0}\mathcal{S}_{L,\sigma}(q_{s})=\int_{0}^{\tau}\left\langle \nabla_{q}L(t,q,\dot{q})-\nabla_{t}(\nabla_{v}L(t,q,\dot{q}))+Y(q)\dot{q},\xi\right\rangle dt.
\]
Thus $\frac{\partial}{\partial s}\Bigl|_{s=0}\mathcal{S}_{L,\sigma}(q_{s})=0$
for all such variations $q_{s}$ if and only if $q$ satisfies the
\textbf{Euler-Lagrange equations }
\begin{equation}
\nabla_{q}L(t,q,\dot{q})-\nabla_{t}(\nabla_{v}L(t,q,\dot{q}))+Y(q)\dot{q}=0.\label{eq:EL}
\end{equation}
Since $\nabla_{vv}L(t,q,v)$ is invertible by \textbf{(L1)}, we can
rewrite this as 
\[
\nabla_{t}\dot{q}=[\nabla_{vv}L(t,q,\dot{q})]^{-1}\left(\nabla_{q}L(t,q,\dot{q})-\nabla_{qv}L(t,q,\dot{q})\dot{q}+Y(q)\dot{q}\right).
\]

In the special case $\sigma=0$, the following theorem is due Abbondandolo
and Schwarz \cite{AbbondandoloSchwarz2009a} (see also \cite[Proposition 3.4.1]{Mazzucchelli2011}
for a detailed proof). However a careful inspection of their proof
reveals that everything still goes through in our setting.
\begin{prop}
\label{prop:morse index}Let $\sigma\in\Omega^{2}(M)$ denote a closed
2-form and $\alpha\in[S^{1},M]$ a $\sigma$-atoroidal class, and
let $L\in C^{\infty}(\mathbb{S}_{\tau}\times TM,\mathbb{R})$ satisfy
\textbf{\emph{(L1) }}and \textbf{\emph{(L2)}}. Then $\mathcal{S}_{L,\sigma}:\mathcal{L}_{\tau}^{\alpha}M\rightarrow\mathbb{R}$
is of class $C^{1}$, and its differential $d\mathcal{S}_{L,\sigma}$
is G\^ateau differentiable and locally Lipschitz continuous. Moreover
its critical points are precisely the (smooth) solutions of the Euler-Lagrange
equation \eqref{eq:EL}, and the second G\^ateau differential $d^{2}\mathcal{S}_{L,\sigma}(q)$
at a critical point $q$ is a Fredholm operator of finite Morse index.
\end{prop}
Recall that a $C^{1}$-functional $\mathcal{S}:\mathcal{M}\rightarrow\mathbb{R}$
on a Riemannian Hilbert manifold $\mathcal{M}$ satisfies the \textbf{Palais-Smale
condition }if every sequence $(q_{m})_{m\in\mathbb{N}}\subseteq\mathcal{M}$
for which $\mathcal{S}(q_{m})$ is bounded and $\left\Vert d\mathcal{S}(q_{m})\right\Vert \rightarrow0$
admits a convergent subsequence (here $\left\Vert \cdot\right\Vert $
denotes the dual norm on $T_{q_{m}}^{*}\mathcal{M}$). The main result
we wish to prove in this Appendix is the following statement. 
\begin{thm}
\label{thm:PS}Let $\sigma\in\Omega^{2}(M)$ satisfy \textbf{\emph{(}}\emph{$\boldsymbol{\sigma}$}\textbf{\emph{1)}},
let $\alpha\in[S^{1},M]$ denote a $\sigma$-atoroidal class, and
let $L\in C^{\infty}(\mathbb{S}_{\tau}\times TM,\mathbb{R})$ satisfy
\textbf{\emph{(L1) }}and \textbf{\emph{(L2)}}. Then there exists $\delta_{0}(L,\sigma,g)>0$
such that if $\tau\left|\delta\right|<\delta_{0}(L,\sigma,g)$ then
$\mathcal{S}_{L,\delta\sigma}:\mathcal{L}_{\tau}^{\alpha}M\rightarrow\mathbb{R}$
satisfies the Palais-Smale condition.
\end{thm}
This theorem was proved for the case $\sigma=0$ originally by Benci
\cite{Benci1986}; our proof however will closely follow that of Abbondandolo
and Figalli \cite[Appendix A]{AbbondandoloFigalli2007}. The proof
of Theorem \ref{thm:PS} makes use of Lemma \ref{lem:BF-2}.
\begin{proof}
(of Theorem \ref{thm:PS})

It follows from \textbf{(L1) }that there exists a constant $D>0$
such that $L(t,q,v)\geq\ell_{0}\left|v\right|^{2}-D$ for all $(t,q,v)\in\mathbb{S}_{\tau}\times TM$.
Thus for any $q\in\mathcal{L}_{\tau}^{\alpha}M$ by Lemma \ref{lem:BF-2}
one has 
\begin{equation}
\mathcal{S}_{L,\delta\sigma}(q)\geq\mathcal{S}_{L}(q)-\left|\mathcal{A}_{\delta\sigma}(q)\right|\geq(\ell_{0}-\left|\delta\right|C_{0}\tau)\left\Vert \dot{q}\right\Vert _{L_{g}^{2}(\mathbb{S}_{\tau})}^{2}-(\left|\delta\right|C_{1}+D).\label{eq:delta small for lag}
\end{equation}
Define 
\begin{equation}
\delta(L,\sigma,g):=\frac{\ell_{0}}{C_{0}},\label{eq:delta L}
\end{equation}
and fix $\delta\in\mathbb{R}$ such that $\tau\left|\delta\right|<\delta(L,\sigma,g)$.
Suppose $(q_{m})_{m\in\mathbb{N}}\subseteq\mathcal{L}_{\tau}^{\alpha}M$
is a sequence such that $\mathcal{S}_{L,\delta\sigma}(q_{m})$ is
bounded and $\left\Vert d\mathcal{S}_{L,\delta\sigma}(q_{m})\right\Vert \rightarrow0$
in the dual norm of $T_{q_{m}}^{*}\mathcal{L}_{\tau}^{\alpha}M$.
Then \eqref{eq:delta small for lag} implies that the sequence $(\dot{q}_{m})$
is bounded in $L_{g}^{2}$. Since 
\[
\mbox{dist}(q_{m}(t),q_{m}(s))\leq\int_{s}^{t}\left|\dot{q}_{m}\right|dr\leq\left|s-t\right|^{1/2}\left\Vert \dot{q}_{m}\right\Vert _{L_{g}^{2}(\mathbb{S}_{\tau})},
\]
the sequence $(q_{m})$ is equicontinuous, and the Arzel\`a-Ascoli
theorem implies that up to passing to a subsequence, we may assume
$q_{m}$ converges uniformly to some $q\in C^{0}(\mathbb{S}_{\tau},M)$. 

We now employ the \textbf{localization }argument of Abbondandolo and
Figalli, which allows us to reduce the problem to one on $\mathbb{R}^{n}$
(roughly speaking, this involves making an intelligent choice of a
chart on $\mathcal{L}_{\tau}^{\alpha}M$ about $q$ - see \cite[Remark 3.4.1]{Mazzucchelli2011}).
As a result, from now on let us assume $L$ is defined on $\mathbb{S}_{\tau}\times U\times\mathbb{R}^{n}$
for some open set $U$ of $\mathbb{R}^{n}$, with $\sigma\in\Omega^{2}(U)$,
and that $(q_{m})\subseteq\mathcal{L}_{\tau}U$ is a sequence such
that $\mathcal{S}_{L,\delta\sigma}(q_{m})$ is bounded and $\left\Vert d\mathcal{S}_{L,\delta\sigma}(q_{m})\right\Vert \rightarrow0$
in the dual norm on $T_{q_{m}}^{*}\mathcal{L}_{\tau}U$, with $(\dot{q}_{m})$
bounded in $L^{2}$ and $q_{m}$ converging uniformly to some $q\in C^{0}(\mathbb{S}_{\tau},U)$. 

This automatically implies that $q\in\mathcal{L}_{\tau}^{\alpha}U$,
and up to passing to a subsequence, $q_{m}$ converges weakly to $q$
in $\mathcal{L}_{\tau}\mathbb{R}^{n}$. To complete the proof we need
to show that this convergence is strong in $W^{1,2}$. Since $(q_{m})$
is bounded in $W^{1,2}$, we have $d\mathcal{S}_{L,\delta\sigma}(q_{m})(q_{m}-q)\rightarrow0$,
and hence by \eqref{eq:differential} (expressed now in the simpler
setting of $\mathbb{R}^{n}$)
\[
\int_{0}^{\tau}\left(\partial_{q}L_{t}(q_{m},\dot{q}_{m})\cdot(q_{m}-q)+\partial_{v}L_{t}(q_{m},\dot{q}_{m})\cdot(\dot{q}_{m}-\dot{q})+\delta Y(q_{m})\dot{q}_{m}\cdot(q_{m}-q)\right)dt\rightarrow0.
\]
The term $\partial_{q}L_{t}(q_{m},\dot{q}_{m})$ is bounded in $L^{2}$
by \textbf{(L2)}. Similarly $Y(q_{m})\dot{q}_{m}$ is bounded in $L^{2}$,
and consequently we have 
\[
\int_{0}^{\tau}\partial_{v}L_{t}(q_{m},\dot{q}_{m})\cdot(\dot{q}_{m}-\dot{q})dt\rightarrow0.
\]
From this it is straightforward to show that $\left\Vert \dot{q}_{m}-\dot{q}\right\Vert _{L_{g}^{2}(\mathbb{S}_{\tau})}\rightarrow0$
using \textbf{(L1) }and \textbf{(L2)}; the proof is identical to \cite[Proposition 3.5.2]{Mazzucchelli2011},
and hence we omit the details. 
\end{proof}
In general the functional $\mathcal{S}_{L,\sigma}$ is \textbf{not
}of class $C^{2}$. In fact, arguing as in \cite[Proposition 3.2]{AbbondandoloSchwarz2009a},
one sees that $\mathcal{S}_{L,\sigma}$ is of class $C^{2}$ if and
only if the function $v\mapsto L(t,q,v)$ is a polynomial of degree
at most 2 for each $(t,q)\in\mathbb{S}_{\tau}\times M$. One would
think that this means that in general there is no hope of doing infinite
dimensional Morse theory with $\mathcal{S}_{L,\sigma}$. Indeed, such
a Morse theory needs at least $C^{2}$-regularity - for example, the
Morse Lemma requires $C^{2}$-regularity - see \cite{Chang1993}.
Nevertheless, under a suitable non-degeneracy assumption (see Condition
\textbf{(N) }below), it \textbf{is }still possible to construct a
\textbf{Morse complex} for $\mathcal{S}_{L,\delta\sigma}$ (see Theorem
\ref{thm:(The-Morse-homology} below). The only missing ingredient
we still need for this is the existence of a \textbf{pseudo-gradient
}for $\mathcal{S}_{L,\sigma}$, which we will discuss shortly in Proposition
\ref{prop:pseudo gradient}.\newline 

The final condition we impose is a non-degeneracy condition:
\begin{lyxlist}{00.00.0000}
\item [{\textbf{(N')}}] Every solution $q$ of the Euler-Lagrange equations
\eqref{eq:EL} is \textbf{non-degenerate}, which means that there
are no nonzero periodic Jacobi fields along $q$.
\end{lyxlist}
Asking for $q$ to be a non-degenerate solution is equivalent to requiring
$q$ to be a non-degenerate critical point of $\mathcal{S}_{L,\sigma}$,
in the sense that the symmetric bilinear form $d^{2}\mathcal{S}_{L,\sigma}(q)$
on $T_{q}\mathcal{L}_{\tau}^{\alpha}M$ is non-degenerate. Moreover
if $H$ is the corresponding Fenchel dual Hamiltonian then $(L,\sigma,\alpha)$
satisfies Condition \textbf{(N') }if and only if\textbf{ }$(H,\sigma,\alpha)$
satisfies Condition \textbf{(N)}.\newline 

The following result can be proved in exactly the same way as \cite[Theorem 4.1]{AbbondandoloSchwarz2009a}.
\begin{prop}
\label{prop:pseudo gradient}Let $\sigma\in\Omega^{2}(M)$ satisfy
\textbf{\emph{(}}\emph{$\boldsymbol{\sigma}$}\textbf{\emph{1)}},
and let $L\in C^{\infty}(\mathbb{S}_{\tau}\times TM,\mathbb{R})$
satisfy \textbf{\emph{(L1) }}and \textbf{\emph{(L2)}}. Fix a $\sigma$-atoroidal
class $\alpha\in[S^{1},M]$, and assume that $(L,\sigma,\alpha)$
satisfies Condition \textbf{\emph{(N')}}. Then there exists a \textbf{\emph{pseudo-gradient}}
for $\mathcal{S}_{L,\sigma}$. That is, there exists a smooth bounded
vector field $\mathcal{G}$ on $\mathcal{L}_{\tau}^{\alpha}M$ whose
zeros are precisely the smooth solutions of the Euler-Lagrange equations
\eqref{eq:EL}, together with a continuous function $\varepsilon\in C(\mathbb{R},\mathbb{R}^{+})$
such that 
\[
d\mathcal{S}_{L,\sigma}(q)\mathcal{G}(q)\geq\varepsilon(\mathcal{S}_{L,\sigma}(q))\left\Vert d\mathcal{S}_{L,\sigma}(q)\right\Vert \ \ \ \mbox{for all }q\in\mathcal{L}_{\tau}^{\alpha}M,
\]
and such that for any solution $q\in \Lambda^{\tau}_{\alpha}M$ of \eqref{eq:EL} one has
\[
d^{2}\mathcal{S}_{L,\sigma}(q)(\xi,\zeta)=\left\langle \nabla\mathcal{G}(q)\xi,\zeta\right\rangle _{W_{g}^{1,2}(\mathbb{S}_{\tau})}\ \ \ \mbox{for all }\xi,\zeta\in W^{1,2}(q^{*}TM)
\]
(here $\nabla\mathcal{G}(q):T_{q}\mathcal{L}_{\tau}^{\alpha}M\rightarrow T_{q}\mathcal{L}_{\tau}^{\alpha}M$
is defined by $\nabla\mathcal{G}(q)\xi:=[\mathcal{G},X](q)$, where
$X$ is any vector field on $\mathcal{L}_{\tau}^{\alpha}M$ such that
$X(q)=\xi$). 
\end{prop}
As mentioned above, Proposition \ref{prop:morse index}, Theorem \ref{thm:PS},
and Proposition \ref{prop:pseudo gradient} imply that one can define
the Morse complex of $\mathcal{S}_{L,\delta\sigma}$ for $\tau\left|\delta\right|<\delta_{0}(L,\sigma,g)$.
We refer the reader to \cite{AbbondandoloMajer2006} for more information
on the construction of the Morse complex, and for the proof of the
following \textbf{Morse homology theorem}. 
\begin{thm}
\textbf{\textup{\label{thm:(The-Morse-homology}}}Let $\sigma\in\Omega^{2}(M)$
satisfy \textbf{\emph{(}}\emph{$\boldsymbol{\sigma}$}\textbf{\emph{1)}},
and let $L\in C^{\infty}(\mathbb{S}_{\tau}\times TM,\mathbb{R})$
satisfy \textbf{\emph{(L1) }}and \textbf{\emph{(L2)}}. Fix a $\sigma$-atoroidal
class $\alpha\in[S^{1},M]$, and fix $\delta\in\mathbb{R}$ such that
$\tau\left|\delta\right|<\delta_{0}(L,\sigma,g)$, and assume that
$(L,\delta\sigma,\alpha)$ satisfies Condition \textbf{\emph{(N')}}.
Denote by $CM_{*}^{\alpha}(L,\delta\sigma,\tau)$ the free $\mathbb{Z}_{2}$-module
generated by the solutions $q$ of the Euler-Lagrange equations \eqref{eq:EL},
graded by their Morse index (as a critical point of $\mathcal{S}_{L,\delta\sigma}$).
Then it is possible to define a map $\partial^{\textrm{\emph{Morse}}}:CM_{*}^{\alpha}(L,\delta\sigma,\tau)\rightarrow CM_{*-1}^{\alpha}(L,\delta\sigma,\tau)$
such that $\partial^{\textrm{\emph{Morse}}}\circ\partial^{\textrm{\emph{Morse}}}=0$,
and such that the associated \textbf{\emph{Morse homology }}
\[
HM_{*}^{\alpha}(L,\delta\sigma,\tau):=H_{*}(CM_{*}^{\alpha}(L,\delta\sigma,\tau);\partial^{\textrm{\emph{Morse}}})
\]
 is isomorphic to the singular homology $H_{*}(\mathcal{L}_{\tau}^{\alpha}M;\mathbb{Z}_{2})$. 
\end{thm}
\bibliographystyle{amsplain}
\bibliography{C:/Users/Will/desktop/willbibtex}

\noindent \emph{Address: }

\noindent (U. Frauenfelder) Department of Mathematics and Research
Institute of Mathematics, Seoul National University San 56-1 Shinrim-dong
Kwanak-gu, Seoul 151-747, Korea 

\noindent (W. J. Merry and G. P. Paternain) Department of Pure Mathematics
and Mathematical Statistics, University of Cambridge, Cambridge CB3
0WB, England\newline 

\noindent \emph{Email:}\texttt{ }

\noindent \texttt{frauenf@snu.ac.kr},\texttt{ w.merry@dpmms.cam.ac.uk},\texttt{
g.p.paternain@dpmms.cam.ac.uk}
\end{document}